\documentclass[reqno,10pt,a4paper]{article}
\usepackage{verbatim}   
\usepackage{color}      
\usepackage{subfigure}  
\usepackage{hyperref}   
\usepackage{amsmath}
\usepackage{algorithm}
\usepackage{setspace}
\usepackage[noend]{algpseudocode}
\usepackage{mathtools}
\usepackage{amsfonts}
\usepackage{amssymb}
\usepackage[pdftex,dvipsnames]{xcolor}  
\usepackage{xargs}                      
\usepackage[pdftex]{graphicx}
\usepackage[colorinlistoftodos,prependcaption,textsize=tiny]{todonotes}
\usepackage[body={16cm,22cm},centering]{geometry}
\usepackage{amsthm,amsfonts,mathrsfs,graphicx}
\usepackage[font=small,labelfont=bf]{caption}
\newtheorem{definition}{Definition}[section]
\newtheorem{theorem}{Theorem}[section]
\newtheorem{lemma}{Lemma}[section]
\newtheorem{remark}{Remark}[section]

{\theoremstyle{definition}
\newtheorem{example}{Example}[section]}

\newtheorem{proposition}{Proposition}[section]

\newcommand{\argmin}{\arg\!\min}
\newcommand{\eqnum}{\refstepcounter{equation}\textup{\tagform@{\theequation}}}

\newcommandx{\info}[2][1=]{\todo[linecolor=OliveGreen,backgroundcolor=OliveGreen!25,bordercolor=OliveGreen,#1]{#2}}
\providecommand{\keywords}[1]{\textbf{\textit{Key words---}} #1}

\algnewcommand\algorithmicinput{\textbf{Input:}}
\algnewcommand\INPUT{\item[\algorithmicinput]}

\makeatletter
\def\namedlabel#1#2{\begingroup
    #2%
    \def\@currentlabel{#2}%
    \phantomsection\label{#1}\endgroup
}

\DeclareMathOperator{\spn}{span}

\DeclareMathOperator{\dom}{dom}
\DeclareMathOperator{\prox}{Prox}

\begin{document}
\title{A Hybrid Finite-Dimensional RHC for Stabilization of Time-Varying Parabolic Equations }
\author{BEHZAD AZMI and KARL KUNISCH}
\date{\today}
\maketitle
\begin{abstract}
The present work is concerned with the stabilization of a general class of time-varying linear parabolic equations by means of a finite-dimensional receding horizon control (RHC).   The stability and suboptimality of the unconstrained receding horizon framework is studied. The analysis allows the choice  of the squared $\ell_1$-norm as control cost. This leads to a nonsmooth  infinite-horizon problem which provides stabilizing optimal controls with a low number of active actuators over time.  Numerical experiments are given which validate the theoretical results and illustrate the qualitative differences between the $\ell_1$- and  $\ell_2$-control costs.
 \end{abstract}


\section{Introduction}
In this work we are concerned with the stabilization of the  controlled system governed by the parabolic equation
\begin{equation}
\label{e1}
\begin{cases}
\partial_t y-\nu\Delta y(t) + a(t)y(t)+ \nabla \cdot (b(t)y(t)) = \sum^N_{i =1} u_i(t)\Phi_i   &\text{ in } (0,\infty)\times\Omega,\\
y =0   & \text{ on } (0,\infty)\times \partial \Omega,\\
y(0)=y_0 &\text{ on } \Omega,
\end{cases}
\end{equation}
with a time depending control vector $\mathbf{u}(t):=[ u_1(t), \dots, u_N(t) ]^t \in L^2(0,\infty;\mathbb{R}^N)$, where $\Omega \subset \mathbb{R}^n$ is a bounded domain with the smooth boundary $\partial \Omega$ and $\nu>0$.  The  functions  $\Phi_i=\Phi_i(x)$ for $i =1, \dots, N$ describe the actuators. The support of these actuators are contained in  an  open subset of  $\Omega$. The reaction  term $a(t)=a(t,x)$ and the  convection term  $b(t)=b(t,x)$ are, respectively, real-  and  $\mathbb{R}^n$-valued functions of $(t,x)\in (0,\infty) \times \Omega $. Although stabilization of the time-varying system of the form \eqref{e1} is of interest on its own,  as a main motivation, we can mention stabilization of  nonlinear controlled systems around the time depending trajectories, see e.g.,  \cite{MR3691212,MR3337988,Phan2018}.  In this case,  the controlled systems of the form \eqref{e1} appears after the linearization of nonlinear systems around a reference trajectory.

Stabilization of the infinite-dimensional controlled systems  by means of finite dimensional controllers have been studied by many authors, see e.g.,  \cite{MR2784695, MR2005128,  MR2974739, MR2104285, MR3691212,MR3337988, Phan2018,MR2629581} and the reference therein.  In all of these
contributions the stabilizing control were given by a feedback control law. In the present work, we construct the stabilizing control within a receding horizon framework. Thus the control objective is to construct a  Receding Horizon Control (RHC)  $\mathbf{u}_{rh}(y_0) \in  L^2(0,\infty;\mathbb{R}^N)$ such that the corresponding state satisfies
\begin{equation*}
\| y_{rh}(t) \|^2_{\mathcal{X}} \leq c_{\mathcal{X}} e^{-\zeta t} \| y_0\|^2_{\mathcal{X}}    \quad  \forall t>0,
\end{equation*}
where  the constants $c_{\mathcal{X}}$ and $\zeta>0$ are independent of $y_0\in \mathcal{X}$.  Here $\mathcal{X}$ will be chosen to be either $H^1_0(\Omega)$ or $L^2(\Omega)$. 

The RHC is constructed through the concatenation of a sequence of open-loop optimal controls on overlapping
temporal intervals covering $[0,\infty)$.  These open-loop subproblems involve a performance index which imposes a structure on the optimal controls.
Here for every $T \in (0, \infty]$ and $0\leq  t_0 \leq T$, we consider
\begin{equation}
\label{e49}
J^p_{T}(\mathbf{u};t_0,y_0):= \frac{1}{2}\int_{t_0}^{t_0+T}\|\nabla y(t)\|^2_{L^2(\Omega)}+\frac{\beta}{2}\int^{t_0+T}_{t_0} |\mathbf{u}(t)|^2_{*}dt
\end{equation}
where the norm $|\cdot |_{*}$ is chosen either as $\ell^2-$norm or $\ell^1-$norm on $\mathbb{R}^N$.  The choice of the $\ell^1-$norm defined by $|\mathbf{u}|_1 = \sum^N_{j=1} |u_j|$ leads to a nonsmooth convex performance index function and enhances sparsity in the coefficient of the control at any $t \in (t_0,t_0+T)$. For every $t>0$ the term $|\mathbf{u}(t)|_{1}$ can be also interpreted as a convex relaxation of  $|\mathbf{u}(t)|_{0}$, see e.g. \cite{MR2243152,MR1963681,MR2045813}. Moreover we can write
\begin{equation}
\label{e4}
\frac{\beta}{2}\int_{t_0}^{t_0+T} |\mathbf{u}(t)|^2_{1}dt = \frac{\beta}{2}\int_{t_0}^{t_0+T} |\mathbf{u}(t)|^2_{2}dt + \beta\int^{t_0+T}_{t_0} \sum^N_{ i,j=1\atop
i<j} |u_i(t)u_j(t)| \,dt.
\end{equation}
The last term in  \eqref{e4} is the $L^1$-penalization of the switching constraint $u_i(t)u_j(t)=0$ for $i\neq j$ and $t>0$. See e.g., \cite{MR3681006,MR3459600}.

Associated to $J^p_{T}$, we consider the following infinite horizon optimal control problem
\begin{equation}
\label{opinf}
\tag*{$OP^{p}_{\infty}(y_0)$}
\min\{ J^p_{\infty}(\mathbf{u};0,y_0) \mid (y,\mathbf{u}) \text{ satisfies } \eqref{e1}, \mathbf{u} \in \mathcal{U}\}.
\end{equation}
For the choice of $ | \cdot |_{*}=| \cdot |_2$, \ref{opinf} is a linear-quadratic problem and one can construct a optimal feedback law based on the corresponding differential Ricatti operator. But, in practice, for the infinite-dimensional controlled systems of the form \eqref{e1}, discretization gives rise to finite-dimensional differential Riccati equations of very large order defined on a relatively large temporal interval. Therefore one is  ultimately confronted with the curse of dimensionality. Further, the choice of $ | \cdot |_{*}=| \cdot |_1$ leads to a nonsmooth  infinite horizon problem.  Finite-horizon optimal control problems with nonsmooth structure have been well-studied for both finite- and infinite-dimensional controlled systems, see e.g., \cite{MR3376312,MR3032866,MR3601024,MR3491785,MR3681006,MR552566,MR2556849,MR2283487}. On the other hand, there is very little research dealing with infinite horizon nonsmmoth problems,  see e.g., \cite{MR3612174,MR3605155}. In \cite{MR3605155}  infinite horizon sparse optimal control problems governed by ordinary differential equations are investigated. In this work, the corresponding sparse optimal controller is approximated by a dynamic programming approach. But again, due to curse of dimensionality, this method is also not feasible for infinite-dimensional time-varying systems. An alternative approach for dealing with \ref{opinf} is the receding horizon framework which allows us to approximate  the solution of nonsmooth infinite horizon problems by a sequence of nonsmooth finite-horizon problems which are well-studied from the theoretical and numerical aspects. The main issue is then to justify the stability of RHC.  Depending on the structure of the underlying problem, this is usually done,
  by techniques involving the design of appropriate sequences of temporal intervals, using  an adequate concatenation scheme, or adding terminal costs and\textbackslash or constraints to the finite horizon subproblems. Due to the structure of the receding horizon framework, the resulting control has a feedback mechanism.

In the present work, we adapt the receding horizon framework proposed  for time-invariant system in \cite{AzmiKunisch} to  time-varying infinite-dimensional linear system.  In this framework, in order to guarantee the stability of RHC, neither terminal costs nor terminal constraints are needed. But rather, by generating an appropriate sequence of overlapping temporal intervals and applying a suitable concatenation scheme, the stability and a certain suboptimality of RHC are obtained.  Previously, this framework was studied for continuous-time finite-dimensional controlled systems in e.g, \cite{MR1833035,MR2950434}  and for discrete-time controlled systems in  e.g, \cite{MR2141559,MR2491596,MR2459581}.

In the RHC approach that we follow here, we choose a sampling time $\delta >0$ and an appropriate prediction horizon $T>\delta$. Then, we define sampling instances $t_k :=k\delta$ for $k=0\dots$.  At every sampling instance $t_k$, an open-loop optimal control problem is solved over a finite prediction horizon $[t_k,t_k+T]$. Then the optimal control is applied to steer the system from time $t_k$ with the initial state $y_{rh}(t_k)$ until time $t_{k+1}:=t_k+\delta$ at which point, a new measurement of state is assumed to be available. The process is repeated starting from the new measured state: we obtain a new optimal control and a new predicted state trajectory by shifting the prediction horizon forward in time. The sampling time $\delta$ is the time period between two sample instances. Throughout, we denote the receding horizon state- and control variables  by $y_{rh}(\cdot)$ and $\mathbf{u}_{rh}(\cdot)$, respectively. Also,  $(y_T^*(\cdot;t_0,y_0), \mathbf{u}^*_T(\cdot;t_0,y_0))$ stands for the optimal state and control of the optimal control problem with  finite time horizon $T$,  and initial function  $y_0$ at initial time $t_0$. This is summarized in  Algorithm \ref{RHA}.

\begin{algorithm}[htbp!]
\caption{Receding Horizon Algorithm}\label{RHA}
\begin{spacing}{1.1}
\begin{algorithmic}[1]
\Require Let the prediction horizon $T$, the sampling time $\delta<T$, and the initial point $y_0\in \mathcal{X}$ be given. Then we proceed through the following steps:
\State $k := 0,\quad t_0 := 0$, and $y_{rh}(t_0):=y_0$.
\State Find the solution $(y_T^*(\cdot;t_k,y_{rh}(t_k)) ,\mathbf{u}^*_T(\cdot;t_k,y_{rh}(t_k)))$ over the time horizon $[t_k,t_k +T]$ by solving the finite horizon open-loop problem
\begin{equation}
\label{e5}
\begin{split}
&\min_{\mathbf{u}\in L^2(t_k,t_k+T;\mathbb{R}^N)}J^p_T(\mathbf{u};t_k,y_{rh}(t_k)):= \int^{t_k+T}_{t_k}(\frac{1}{2}\|\nabla y(t)\|^2_{L^2(\Omega)}+\frac{\beta}{2} |\mathbf{u}(t)|^2_{*})dt  \\
\text{ s.t } &\begin{cases}
\partial_t y(t)-\nu\Delta y(t) + a(t)y(t)+ \nabla \cdot (b(t)y(t)) = \sum^N_{i =1} u_i(t)\Phi_i  &\text{ in } (t_k,t_k+T)\times\Omega,\\
y =0   & \text{ on } (t_k,t_k+T)\times \partial \Omega,\\
y(t_k)=y_{rh}(t_k) &\text{ on } \Omega,
\end{cases}
\end{split}
\end{equation}
\State Set \begin{align*}
\mathbf{u}_{rh}(\tau)&:=\mathbf{u}^*_T(\tau;t_k,y_{rh}(t_k)) \quad &\text{ for all } \tau \in [t_k,t_k+\delta),\\
y_{rh}(\tau)&:=y^*_T(\tau;t_k,y_{rh}(t_k)) \quad &\text{ for all } \tau \in [t_k,t_k+\delta],\\
t_{k+1} &:= t_k +\delta, & \\
k &:= k+1.&
\end{align*}
\State Go to step 2.
\end{algorithmic}
\end{spacing}
\end{algorithm}

In the light of our recent investigations on analysis of RHC for infinite-dimensional systems  in \cite{MR3843184,MR3721863,AzmiKunisch}, the novelty of the present paper lies in the following facts: 1. Here we deal with time-varying systems.  2. Particularly in comparision to our previous investigation in \cite{AzmiKunisch},  we study the stability of RHC for the $H_0^1(\Omega)$-tracking term in the performance index function. Based on an observability inequality, we will show the exponential stability of RHC which was not the case for $L^2(\Omega)$-tracking term in \cite{AzmiKunisch}.   Further, we will see that, for more regular data, the stabilization (with respect to $H^1_0$-norm ) of the strong solution holds with the same rate as for the weak solution.  3. Here our RHC consists of finite-dimensional time-dependent controllers. 4. By incorporating the  squared $\ell_1$-norm as the control cost, we demonstrate the sparse controls can also be treated in the RHC framework, both analytically and numerically.

The remainder of the paper is organized as follows: In Section 2, the stability and suboptimality of RHC is investigated for a general abstract time-varying linear controlled system for which system \eqref{e1} counts as a special case. Sections 3 reviews some facts  about well-posedness and regularity of the solution to \eqref{e1}.  Section 4 deals with well-posedness and first-order optimality conditions of  the open-loop subproblems. Further, in the 5-th section  selected results on stabilizability of \eqref{e1} by finitely many controllers are summarized. Then, the main results i.e., the asymptotic stability and suboptimality of \eqref{e1} according to the regularity of the solution and the choice of performance index function are given in Section 6. Section 7, contains the numerical experiments which validate the theoretical results in the previous sections and  illustrate the qualitative differences between the $\ell_1$- and  $\ell_2$-control costs.

\section{Stability of the receding horizon control}
This section is devoted to investigating the stability of RHC for nonautonomous systems in an abstract framework which contains the above discussion as a special case. Let $V \hookrightarrow  H=H' \hookrightarrow  V'$ be a Gelfand triple of real Hilbert spaces with $V$ densely  contained in $H$. Further let $U$  denote the control space which is assumed to be a real Hilbert space. For any $T \in \mathbb{R}_+ \cup \{ \infty \}$, $t_0\geq 0$, and $y_0 \in H$,  consider the time-varying linear system
\begin{equation}
\label{e6}
\tag*{$LTV(T,t_0,y_0)$}
\begin{cases}
\partial_ty(t)=A(t)y(t)+B(t)\mathbf{u}(t) &\quad \text{ for } t \in  ( t_0,t_0+T)\\
      y(t_0)=y_0, &
\end{cases}
\end{equation}
where $A(t) \in \mathcal{L}(V,V')$ and $B(t) \in \mathcal{L}(U,V')$ for almost every $t \in  ( t_0,t_0+T)$. Throughout the section, it is assumed that for any quadruple $(T,t_0,y_0,\mathbf{u}) \in \mathbb{R}^2_+\times H \times L^2(t_0,t_0+T;U)$ with a finite $T>0$, equation \ref{e6} admits a unique solution $y^{\mathbf{u}} \in W(t_0,t_0+T;V,V')$ satisfying
\begin{equation*}
y^{\mathbf{u}}(t)-y_0=\int_{t_0}^{t}(A(s)y^{\mathbf{u}}(s)+B(s)\mathbf{u}(s))ds\quad \text{ in } V'
\end{equation*}
for $t \in [t_0,t_0+T]$, where
\begin{equation}
\label{wnullt}
W(t_0,t_0+T;V,V'):= L^2(t_0,t_0+T;V)\cap H^1(t_0,t_0+T;V'),
\end{equation}
is endowed with the norm $\|v\|_{W(t_0,t_0+T;V,V')}: =  (\|\partial_t v\|^2_{L^2(t_0,t_0+T;V')}+ \| v\|^2_{L^2(t_0,t_0+T;V)})^{\frac{1}{2}}$.  We recall that $W(t_0,t_0 +T;V,V')$  is continuously embedded in $C([t_0,t_0 +T];H)$,  see e.g. \cite{wlokapartial,temam1997infinite}. Moreover, for every finite $T$ and  the solution $y^{\mathbf{u}}$ we shall require the estimate
\begin{equation}
\label{Est1}
\|y^\mathbf{u}\|^2_{C([t_0,t_0 +T];H)} \leq c_T \left(\|y_0\|^2_H +  \| \mathbf{u}\|^2_{L^2(t_0,t_0+T;U)}\right),
\end{equation}
where the constant $c_T$ is independent of $y_0$, $f$, and $\mathbf{u}$. Further, $c_T$ may increase exponentially as $T\to \infty$.

For the choice  $A(t)y =(-\nu\Delta+a(t))y + \nabla \cdot (b(t)y)$, $B(t):=\left[ \Phi_1, \dots, \Phi_N \right]$, and  $U:=\mathbb{R}^N$,  the controlled system \eqref{e1} is a special case of \ref{e6}.

To specify our optimal control problems, we introduce the incremental function $\ell: \mathbb{R}_+ \times V\times U \to \mathbb{R}_+$ satisfying
\begin{equation}
\label{estiob}
\ell(t,y,\mathbf{u})\geq \alpha_{\ell}( \|y\|^2_{H}+\|\mathbf{u}\|^2_{U}) \quad   \text{ for every }  t\geq 0  \text{  and every }  (y,\mathbf{u}) \in  V \times U,
 \end{equation}
where  $\alpha_{\ell}>0$ is independent of $(t,y,\mathbf{u})$, and $\ell(t,0,0) =0$ for every $t \in \mathbb{R}_+$.

For a given prediction horizon of length $T>0$, and initial state $y_0 \in H$ at time $t_0$, the receding horizon approach relies on the finite horizon optimal control problem of the form
\begin{equation}
\label{PT}
\tag*{$OP_{T}(t_0,y_0)$}
\min_{\mathbf{u}\in L^2(t_0,t_0+T;U)}J_T(\mathbf{u};t_0,y_0):= \int^{t_0+T}_{t_0} \ell(t,y(t),\mathbf{u}(t))dt
\text{  subject  to  \ref{e6}}.
\end{equation}
The solution to  \ref{PT} and its associated state will be denoted by $(y^*_T(t;t_0,y_0), \mathbf{u}^*_T(t,t_0,y_0))$.  The receding horizon technique will be used to solve the following infinite horizon problem
\begin{equation}
\label{opinf1}
\tag*{$OP_{\infty}(y_0)$}
 \min_{ \mathbf{u} \in L^2(0,\infty;U)}\{J_{\infty}(\mathbf{u};0,y_0) \text{ subject to $LVT(\infty,0,y_0)$}  \}.
\end{equation}
This technique can be expressed as in Algorithm \ref{RHA2}.
\begin{algorithm}[htbp!]
\caption{Receding Horizon Algorithm for abstract system}\label{RHA2}
\begin{spacing}{1.1}
\begin{algorithmic}[1]
\Require Let the prediction horizon $T$, the sampling time $\delta<T$, and the initial point $y_0\in H$ be given. 
\Comment{ We proceed through the steps of Algorithm \ref{RHA} except that Step 2  is replaced by:}

\mbox{2. Find $(y_T^*(\cdot;t_k,y_{rh}(t_k)) ,\mathbf{u}^*_T(\cdot;t_k,y_{rh}(t_k)))$ over $[t_k,t_k +T]$ by solving $OP_{T}(t_k,y_{rh}(t_k))$}.
\end{algorithmic}
\end{spacing}
\end{algorithm}
\begin{definition}
For any $y_0 \in H$  the infinite horizon value function $V_{\infty}: H \to \mathbb{R}_+$ is defined by
\begin{equation*}
V_{\infty}(y_0):= \min_{ \mathbf{u} \in L^2(0,\infty;U)}\{J_{\infty}(\mathbf{u};0,y_0) \text{ subject to $LVT(\infty,0,y_0)$}  \}.
\end{equation*}
Similarly, for every $(T,t_0,y_0) \in \mathbb{R}^2_+ \times H$,  the finite horizon value function $V_{T}: \mathbb{R}_+ \times H \to \mathbb{R}_+$ is defined by
\begin{equation*}
V_{T}(t_0,y_0):= \min_{\mathbf{u} \in L^2(t_0,t_0+T;U)}\{J_{T}(\mathbf{u};t_0,y_0 ) \text{ subject to \ref{e6}} \}.
\end{equation*}
\end{definition}
In order to show the exponential stability and suboptimality of the receding horizon control obtained by Algorithm \ref{RHA2},  we shall need to verify the following properties:
\begin{description}
\item[\namedlabel{P1}{P1}] For every $(T,t_0,y_0)\in  \mathbb{R}^2_+ \times H$,  every finite horizon optimal control problem of the form  \ref{PT}  admits a solution.
 \end{description}
 Moreover, we require the following properties for the finite horizon value function $V_T$:
\begin{description}
\item[\namedlabel{P2}{P2}] For every positive number $T$, $V_T$ is \textit{globally decrescent} with respect to the $H$-norm. That is, there exists a continuous,  non-decreasing,  and bounded function $\gamma_2: \mathbb{R}_+ \to \mathbb{R}_+$ such that
\begin{equation}
\label{e7}
 V_T(t_0,y_0)  \leq  \gamma_2(T)\|y_0\|^2_{H} \quad \text{ for every } (t_0,y_0)\in   \mathbb{R}_+ \times H.
\end{equation}
\item[\namedlabel{P3}{P3}] For every $T>0$,  $V_T$ is \textit{uniformly positive} with respect to the $H$-norm. In other words,  for every $T>0$ there exists a constant $\gamma_1(T)>0$ such that we have
\begin{equation}
\label{e8}
V_T(t_0,y_0)  \geq \gamma_1(T)\|y_0\|^2_{H}  \quad \text{ for every } (t_0,y_0)\in \mathbb{R}_+ \times H.
\end{equation}
\end{description}
\begin{remark}
The constant $\gamma_1(T)$ is related to the observability inequalities for the linear system \ref{e6}  with $\mathbf{u} = 0$. For infinite-dimentional  parabolic and hyperbolic systems, we refer to \cite{MR2383077} and for finite dimensional systems we mention the reference \cite{MR1844565}. Later, we will see that, for the control system \eqref{e1} with the performance index function \eqref{e49}, we have  $\gamma_1(T) \to 0$ monotonically as $T\to 0$.
\end{remark}

For the sake of simplicity, throughout this section,  we use the notation
\begin{equation*}
 \ell^*_T(t;t_0,y_0) := \ell(t,y^*_T(t;t_0,y_0),\mathbf{u}^*_T(t,t_0,y_0))    \quad \text{ for every } t \in (t_0,t_0+T).
\end{equation*}
The results  of this section are similar to the those in \cite{AzmiKunisch, MR3721863} with the difference that here the dynamical system is nonautonomous and, thus, the finite horizon value function depends also on initial time and the estimates are not translation invariant.  For the sake of completeness and  the convenience of the reader, we provide and adapt some the proofs here.
\begin{lemma}\label{lem2}
 If \ref{P1}-\ref{P2} hold and $T>\delta>0$ are given,  then for every  $(t_0,y_0) \in  \mathbb{R}_+ \times H$ the following inequalities hold:
\begin{equation}
\label{lem2c1}
\begin{split}
&V_T(t_0+\delta, y_T^*(t_0+\delta;t_0,y_0)) \\
 &\leq\int^{t_0+t^*}_{t_0+\delta} \ell^*_T(t;t_0,y_0)dt+\gamma_2(T+\delta-t^*)\| y_T^*(t_0+ t^*;t_0,y_0)\|^2_{H}\quad \text{for all } t^* \in [\delta, T],
\end{split}
\end{equation}
 and
\begin{equation}
\label{lem2c2}
\int^{t_0+T}_{t_0+t^*} \ell^*_T(t;t_0,y_0)dt \leq\gamma_2(T-t^*)\| y_T^*(t_0+t^*;t_0,y_0)\|^2_{H} \quad \text{for all }t^* \in [0, T].
\end{equation}
\end{lemma}
\begin{proof}
Due to Bellman's optimality principle and utilizing \ref{P1} and \ref{P2}, we have for every $t^* \in [\delta, T]$
\begin{equation}
\label{e11}
\begin{split}
V_T&( t_0+\delta,y_T^*(t_0+\delta;t_0,y_0))  \\
& = \min_{\mathbf{u} \in L^2(t_0+\delta, t_0+t^*;U)}\left\lbrace \int^{t_0+t^*}_{t_0+\delta} \ell(t,y^{\mathbf{u}}(t), \mathbf{u}(t))dt+ V_{T+\delta-t^*}(t_0+t^*, y^{\mathbf{u}}(t_0+t^*)) \right\rbrace  \\
&\leq \int^{t_0+t^*}_{t_0+\delta} \ell^*_T(t;t_0,y_0)dt+V_{T+\delta-t^*}(t_0+t^*, y_T^*(t_0+t^*;t_0,y_0)) \\
&\leq \int^{t_0+t^*}_{t_0+\delta} \ell^*_T(t;t_0,y_0)dt + \gamma_2(T+\delta-t^*)\| y_T^*(t_0+t^*;t_0,y_0)\|^2_{H},
\end{split}
\end{equation}
where $y^{\mathbf{u}}$ in the above equality is the solution to $LTV(T,t_0,y_0)$ for $\mathbf{u} \in L^2(t_0+\delta, t_0+t^*;U)$.

To prove the second inequality let  $t^* \in [0,T]$  be given. Similarly, to the first inequality, using  Bellman's principle and \eqref{e7}, we have
\begin{equation*}
\int^{t_0+T}_{t_0+t^*} \ell^*_T(t;t_0,y_0)dt = V_{T-t^*}(t_0+t^*, y_T^*(t_0+t^*;t_0,y_0))  \leq\gamma_2(T-t^*)\| y_T^*(t_0+t^*;t_0,y_0)\|^2_{H},
\end{equation*}
as desired.
\end{proof}
\begin{lemma}
\label{lem3}
Suppose that \ref{P1} and \ref{P2} hold. Then for given  $(T,\delta,t_0,y_0) \in  \mathbb{R}^3_+ \times H$  with $T>\delta$ and  the choice
\begin{equation*}
\theta_1(T,\delta) := 1+\frac{\gamma_2(T)}{\alpha_{\ell}(T-\delta)},  \qquad \theta_2(T,\delta) := \frac{\gamma_2(T)}{\alpha_{\ell}\delta},
\end{equation*}
we have the following estimates
\begin{equation}
\label{e23}
V_T(t_0+\delta,y_T^*(t_0+\delta;t_0,y_0)) \leq \theta_1\int^{t_0+T}_{t_0+\delta} \ell^*_T(t;t_0,y_0)dt,
\end{equation}
and
\begin{equation}
\label{e24}
\int^{t_0+T}_{t_0+\delta} \ell^*_T(t;t_0,y_0)dt \leq \theta_2 \int^{t_0+\delta}_{t_0} \ell^*_T(t;t_0,y_0)dt.
\end{equation}
\end{lemma}
\begin{proof}
 To verify \eqref{e23}  recall that $ y_T^*(\cdot;y_0,t_0) \in C([t_0,t_0+T];H)$. Hence  there is a $\bar{t}\in [\delta, T]$ such that
\begin{equation*}
\bar{t}= \argmin_{t \in [\delta, T]}\| y_T^*(t_0+t ; t_0,y_0) \|^2_{H}.
\end{equation*}
By \eqref{estiob}  and  \eqref{lem2c1}, we have
\begin{equation}
\label{e15}
\begin{split}
&V_T(t_0+\delta, y_T^*(t_0+\delta;t_0,y_0)) \stackrel{\text{\eqref{lem2c1}}}{\leq}\int^{t_0+\bar{t}}_{t_0+\delta} \ell^*_T(t;t_0,y_0)dt+\gamma_2(T+\delta-\bar{t})\| y_T^*(t_0+\bar{t};t_0,y_0)\|^2_{H}\\
& \leq \int^{t_0+\bar{t}}_{t_0+\delta} \ell^*_T(t;t_0,y_0)dt+\gamma_2(T)\| y_T^*(t_0+\bar{t};t_0,y_0)\|^2_{H}\\
& \leq \int^{t_0+T}_{t_0+\delta} \ell^*_T(t;t_0,y_0)dt+\frac{\gamma_2(T)}{T-\delta}\int^{t_0+T}_{t_0+\delta}\| y_T^*(t;t_0,y_0)\|^2_{H}dt,\\
& \leq \int^{t_0+T}_{t_0+\delta} \ell^*_T(t;t_0,y_0)dt+\frac{\gamma_2(T)}{\alpha_{\ell}(T-\delta)}\int^{t_0+T}_{t_0+\delta}\ell^*_T(t;t_0,y_0)dt  =(1+\frac{\gamma_2(T)}{\alpha_{\ell}(T-\delta)})\int^{t_0+T}_{t_0+\delta}\ell^*_T(t;t_0,y_0)dt,
\end{split}
\end{equation}
which implies \eqref{e23}.
Turning to  \eqref{e24} we define
\begin{equation*}
\hat{t}= \argmin_{t \in [0,\delta]}\| y_T^*(t_0 +t ; t_0,y_0)\|^2_{H}.
\end{equation*}
Then by \eqref{estiob}  and  \eqref{lem2c2}, we have
\begin{equation}
\label{e155}
\begin{split}
&\int^{t_0+T}_{t_0+\delta} \ell^*_T(t;t_0,y_0)dt \leq \int^{t_0+T}_{t_0+\hat{t}} \ell^*_T(t;t_0,y_0)dt \leq \gamma_2(T-\hat{t})\| y_T^*(t_0+\hat{t};t_0,y_0)\|^2_{H}\\
&\leq \gamma_2(T)\| y_T^*(t_0+\hat{t};t_0,y_0)\|^2_{H} \leq\frac{\gamma_2(T)}{\delta}\int^{t_0+\delta}_{t_0}\| y_T^*(t;t_0,y_0)\|^2_{H}dt\leq \frac{\gamma_2(T)}{\alpha_{\ell}\delta}\int^{t_0+\delta}_{t_0} \ell^*_T(t;t_0,y_0)dt,
\end{split}
\end{equation}
which provides \eqref{e24}.
\end{proof}

\begin{proposition}
\label{pro1}
Suppose that \ref{P1}-\ref{P2} hold and let $\delta>0$ be given. Then there exist $T^*>\delta$ and $\alpha \in (0,1)$ such that  the following inequality is satisfied
\begin{equation}
\label{e20s}
V_T(t_0+\delta ,y^*_T(t_0+\delta;t_0,y_0)) \leq V_T(t_0,y_0)-\alpha \int_{t_0}^{t_0+\delta} \ell^*_T(t;t_0,y_0)dt
\end{equation}
for every $T\geq T^*$ and $(t_0,y_0) \in \mathbb{R}_+ \times H$.
\end{proposition}

\begin{proof}
From the definition of $V_T(t_0,y_0)$  we have
\begin{equation*}
\begin{split}
V_T&(t_0+\delta ,y^*_T(t_0+\delta;t_0,y_0)) - V_T(t_0,y_0)=V_T(t_0+\delta ,y^*_T(t_0+\delta;t_0,y_0)) -\int^{t_0+T}_{t_0}\ell^*_T(t;t_0,y_0)dt\\
              \leq&(\theta_1-1)\int^{t_0+T}_{t_0+\delta}\ell^*_T(t;t_0,y_0)dt-\int^{t_0+\delta}_{t_0}\ell^*_T(t;t_0,y_0)dt
              \leq(\theta_2(\theta_1-1)-1)\int^{t_0+\delta}_{t_0}\ell^*_T(t;t_0,y_0)dt,
\end{split}
\end{equation*}
where $\theta_1$ and $\theta_2$ are defined in Lemma \ref{lem3}. Since
\begin{equation}
\label{e3}
 \alpha(T) := 1-\theta_2(T)(\theta_1(T)-1) =1- \frac{\gamma_2^2(T)}{\alpha^2_{\ell}\delta(T-\delta)} \to 1 \text{ for } T\to \infty,
\end{equation}
there exist $T^*>\delta$ and $\alpha(T^*) \in (0,1)$ such that $1-\theta_2(T)(\theta_1(T)-1) \geq \alpha(T^*)$ for all
$T\geq T^*$. This implies \eqref{e20s}.
 \end{proof}
\begin{theorem}[Suboptimality and exponential decay]
\label{subopth}
Suppose that \ref{P1}-\ref{P2} hold, and let a sampling time $\delta>0$ be given. Then there exist numbers $T^* > \delta$, and $\alpha \in (0,1)$,  such that for every fixed  prediction horizon $T \geq T^*$, and every $y_0 \in H$ the receding horizon control $\mathbf{u}_{rh}$ obtained from  Algorithm \ref{RHA2} satisfies the \textbf{suboptimality} inequality
\begin{equation}
\label{ed27}
\alpha V_{\infty}(y_0) \leq \alpha J_{\infty}(\mathbf{u}_{rh};0,y_0)\leq V_T(0,y_0) \leq V_{\infty}(y_0).
\end{equation}
If additionally \ref{P3} holds we have  \textbf{exponential stability}
\begin{equation}
\label{ed28}
\|y_{rh}(t)\|^2_{H} \leq c_He^{-\zeta t}\|y_0\|^2_{H} \quad  \text{ for }  t\geq 0,
\end{equation}
where the positive numbers $\zeta$ and  $c_H$  depend on $\alpha$, $\delta$, and $T$, but are independent of $y_0$.
\end{theorem}
Utilizing the previous lemmas the proof of this result follows the lines of the verification of \cite[Theorem 1.5]{MR3721863}.  But since we refer to it on several occasions it is provided in Appendix \ref{apend1}.

\begin{remark}
{\em For fixed $\delta >0$, due to inequality \eqref{e3} we have $\lim_{T \to \infty} \alpha(T) = 1$. Thus RHC is asymptotically optimal. Moreover, for fixed $T\geq T^*$
we obtain  that $\alpha \to -\infty$ as $\delta \to 0$. That is,  for arbitrarily small sampling times  $\delta$,  the suboptimality and asymptotic stability of RHC
is not guaranteed.}
\end{remark}

\section{Well-posedness and regularity of solutions}
In this section we are back to the concrete problem \ref{opinf} governed by \eqref{e1}. To summarize useful  well-posedness and regularity properties we first consider
\begin{equation}
\label{e17}
\begin{cases}
\partial_t y(t)-\nu\Delta y(t) + a(t)y(t)+ \nabla \cdot (b(t)y(t)) = f(t)   &\text{ in } (t_0,t_0+T)\times\Omega,\\
y =0   & \text{ on } (t_0,t_0+T)\times \partial \Omega,\\
y(t_0)=y_0 &\text{ on } \Omega.
\end{cases}
\end{equation}
We set $H:=L^2(\Omega ; \mathbb{R})$, $V:=H^1_0(\Omega; \mathbb{R})$, and $V':=H^{-1}(\Omega;\mathbb{R})$,  and endow $V$ with the following scalar product and corresponding norm
\begin{equation*}
(\phi,\psi)_V:=(\nabla \phi ,\nabla \psi)_H,   \quad   \|\phi \|_V :=(\phi,\phi)^{\frac{1}{2}}_V= \|\nabla \phi \|_H    \quad   \text{ for every  }  \phi, \psi \in V.
\end{equation*}
 Throughout it is assumed that
 \begin{equation}
 \label{e56}
 \tag{RA}
 a \in L^{\infty}(0, \infty;  L^r(\Omega;\mathbb{R})) \text{ with } r\geq n :=dim(\Omega), \text{ and } b \in L^{\infty}((0,\infty)\times \Omega ; \mathbb{R}^n),
 \end{equation}
 and we set
 \begin{equation*}
 N(a,b):= \| a\|_{L^{\infty}(0,\infty;L^r(\Omega))}+\| b\|_{L^{\infty}((0,\infty)\times \Omega)}.
 \end{equation*}
 We recall the  notion of weak variational solution for \eqref{e17}:
\begin{definition}
Let  $(T,t_0, y_0, f ) \in   \mathbb{R}^2 \times  H \times L^2(t_0,t_0+T;V') $ be given.  Then,  a function $y \in W(t_0,t_0+T;V,V')$ is referred to as a weak solution of \eqref{e17} if for almost every $t \in (t_0, t_0+T)$ we have
\begin{equation}
\label{e19}
\langle \partial_t y(t),\varphi \rangle_{V',V}+\nu( \nabla y(t),\nabla \varphi)_{H}+ \langle a(t)y(t),\varphi \rangle_{V',V} - (b(t)y(t), \nabla \varphi)_H   =\langle f(t),\varphi \rangle_{V',V} \quad \text{ for all }\varphi \in V,
\end{equation}
and $y(t_0)=y_0$ is satisfied in $H$.
\end{definition}
In the following we present the well-posedness of weak solutions to
\eqref{e17}, as well as an observability type inequality,  which will be essential to derive the exponential stability for RHC.
\begin{proposition}
\label{Theo2}
For every multiple  $(T,t_0, y_0, f ) \in   \mathbb{R}^2 \times  H \times L^2(t_0,t_0+T;V')$ equation \eqref{e17} admits a unique weak solution $y \in W(t_0,t_0+T;V,V')$ satisfying
\begin{equation}
\label{e29}
\|y\|^2_{C([t_0,t_0+T];H)} + \| y \|^2_{W(t_0,t_0+T;V,V')} \leq c_1\left( \|y_0\|^2_H + \| f\|^2_{L^2(t_0,t_0+T;V')} \right), 
\end{equation}
with  $c_1$ depending on  $(T,\nu,a,b,\Omega)$. Moreover, we have the following observability inequality
\begin{equation}
\label{e14}
\|y_0\|^2_H \leq  \hat{c}_{\nu} \left(1 + \frac{1}{T} + N(a,b)\right)\|y\|^2_{L^2(t_0,t_0+T;V)}+ \|f\|^2_{L^2(t_0,t_0+T;V')},
\end{equation}
with  $\hat{c}_{\nu}$ depending only on $(\nu, \Omega)$.
\end{proposition}

The proof can be given by standard estimates and is therefore deferred to  Appendix \ref{apend2}.

In order to show the stabilizability of RHC with respect to the $V$-norm we need the notion of the strong solution. Introducing $D(A):= H^2(\Omega)\cap V$, we have the following relations
\begin{equation}
\label{e35}
D(A) \hookrightarrow V \hookrightarrow H=H' \hookrightarrow V' \hookrightarrow D(A)'.
\end{equation}
For any interval $(t_0,t_0+T)$ with $T\in \mathbb{R}_+ \cup \{\infty\}$, we consider
\begin{equation*}
W(t_0,t_0+T;D(A),H):=L^2(t_0,t_0+T;D(A)) \cap H^1(t_0,t_0+T;H),
\end{equation*}
endowed with the norm
\begin{equation*}
\|y\|_{W(t_0,t_0+T;D(A),H)}:=\left(  \| y\|^2_{L^2(t_0,t_0+T;D(A))}+\| \partial_ty\|^2_{L^2(t_0,t_0+T; H)} \right)^{\frac{1}{2}},
\end{equation*}
as the space for strong solutions.  Based on  \eqref{e35}, it is known that $W(t_0,t_0+T;D(A),H)\hookrightarrow C([t_0,t_0+T];V)$, see e.g., \cite{MR0350177}[Chapter 3, Section 1.4 ] and  \cite{MR769654}. Then we have the following notion of strong solution:
\begin{definition}[Strong solution]
A weak solution to \eqref{e17} is called a \textbf{strong} solution, provided that it belongs to  $W(t_0,t_0+T;D(A),H)$.
\end{definition}
In order to obtain the strong solutions for \eqref{e17}, we need to impose the following additional regularity condition on the convection term $b$:
\begin{equation}
\tag{SRA}
\label{e36}
\nabla\cdot b \in L^{\infty}(0,\infty;L^d(\Omega)) \text{ with } d \text{ satisfying  } \begin{cases}  d =2, & \text{ if }  n\in \{1,2,3\},\\                                                                                                                                 d \geq \frac{2n}{3} &   \text{ if } n\geq 4. \end{cases}
\end{equation}
Later, we will use the notation
\begin{equation*}
\tilde{N}(a,b):=\| a\|_{L^{\infty}(0,\infty;L^r(\Omega))}+\| b\|_{L^{\infty}((0,\infty)\times \Omega)}+\|\nabla \cdot b\|_{L^{\infty}(0,\infty;L^r(\Omega))}.
\end{equation*}
In the next theorem, we present the existence result for the strong solution to \eqref{e17}.
\begin{proposition}
\label{Theo4}
Assume that \eqref{e36} holds. Then for every quadruple  $(T,t_0, y_0, f ) \in   \mathbb{R}^2_+ \times  V \times L^2(t_0,t_0+T;H)$, equation \eqref{e17} admits a unique strong solution $y \in W(t_0,t_0+T ; D(A),H)$ satisfying
\begin{equation}
\label{e43}
\|y\|^2_{C([t_0,t_0+T];V)}+ \| y\|^2_{W(t_0,t_0+T;D(A),H)} \leq c_2\left(\| y_0\|^2_V + \| f\|^2_{L^2(t_0,t_0+T;H)} \right),
\end{equation}
where the constant $c_2$  depends on $(T,\nu, \tilde{c}_{a,b}, \Omega)$.
\end{proposition}
\begin{proof}
Similarly to the proof of Proposition \ref{Theo2}, the proof uses Galerkin approximations.  We rely on subsequences which converge weakly  in $L^2(t_0,t_0+T;D(A))$ and weakly-star in $L^{\infty}(t_0,t_0+T;V)$. To show this, we need to derive some a-priori estimates. Throughout, $c>0$ is a generic constant  that depends only on $\Omega$ and $\nu$.  Assume that  $y$ is regular enough.  By multiplying  equation \eqref{e17} by $-\Delta y(t)$,  we can write for almost every $t \in (t_0,t_0+T)$ that
\begin{equation}
\label{e38}
\frac{d}{2dt}\| \nabla y(t)\|^2_H +\nu\| \Delta y(t)\|^2_H \leq  |(a(t)y(t),\Delta  y(t))_H| + |(\nabla \cdot (b(t)y(t)),\Delta y(t))_H|+| (f(t),\Delta y(t))_H|.
\end{equation}
 Using $\eqref{e37}$,  we have
\begin{equation}
\label{e39}
|(a(t)y(t),\Delta y(t))_H | \leq \|a\|_{L^{\infty}(0,\infty ; L^r(\Omega))}\|y(t)\|_V\|\Delta y(t)\|_H  \text{ for almost every } t \in (t_0,t_0+T).
\end{equation}
Moreover, using  the fact that  $\nabla \cdot (b(t)y(t))= (\nabla\cdot b(t))y(t)+ b(t) \cdot\nabla y(t)$,  we obtain
\begin{equation}
\label{e40}
\begin{split}
|(\nabla \cdot (b(t)y(t)),\Delta y(t))_H|& \leq |(\nabla\cdot b(t))y(t), \Delta  y(t))_H| + |(b(t) \cdot\nabla y(t),\Delta  y(t))_H|.
\end{split}
\end{equation}
Now, we consider the two cases $n\leq 3$ and $n\geq 4$ separately.  For the case that $1\leq  n \leq 3$, due to  Agmon's inequality \cite{MR2589244}[Lemma 13.2] we have $D(A) \hookrightarrow L^{\infty}(\Omega)$ and, thus, we can write for almost every $t \in (t_0,t_0+T)$ that
\begin{equation}
\label{e41a}
|(\nabla\cdot b(t))y(t), \Delta  y(t))_H|   \leq  \|(\nabla\cdot b)(t)\|_{H}\| y(t)\|_{L^{\infty}(\Omega)}\|  \Delta  y(t)\|_H \leq c \|(\nabla\cdot b)(t)\|_{H}\|y(t)\|^{\frac{1}{2}}_{V}\|\Delta y(t)\|^{\frac{3}{2}}_{H}.
\end{equation}
Whereas, for the case $n\geq 4$,  due to the  Gagliardo-Nirenberg interpolation inequality \cite{MR2527916},  we have $D(A) \hookrightarrow L^{\frac{2n}{n-3}}(\Omega)$  with
\begin{equation*}
\|y(t) \|_{L^{\frac{2n}{n-3}}(\Omega)} \leq  c \left( \| y(t)\|^{\frac{1}{2}}_{L^{\frac{2n}{n-2}}(\Omega)} \|\Delta  y(t)\|^{\frac{1}{2}}_{H}\right),
\end{equation*}
and, as a consequence, due to the fact that $V \hookrightarrow L^{\frac{2n}{n-2}}(\Omega)$, we can write
\begin{equation}
\begin{split}
\label{e41b}
|(\nabla\cdot b(t))y(t), \Delta  y(t))_H|  &\leq \|(\nabla\cdot b)(t)\|_{L^{\frac{2n}{3}}(\Omega)} \|y(t) \|_{L^{\frac{2n}{n-3}}(\Omega)}\| \Delta  y(t)\|_{H} \\
 & \leq  c\|(\nabla\cdot b)(t)\|_{L^d(\Omega)} \|y(t)\|^{\frac{1}{2}}_{V}\|\Delta  y(t)\|^{\frac{3}{2}}_{H},
\end{split}
\end{equation}
Then, due to \eqref{e29}, \eqref{e38}, \eqref{e39}, \eqref{e40}, \eqref{e41a}, \eqref{e41b}, and  Young's and Gronwall's inequalities,  we obtain
\begin{equation}
\label{e41}
\|y\|^2_{L^{\infty}(t_0,t_0+T;V)}+ \nu\| y\|^2_{L^2(t_0,t_0+T;D(A))} \leq {c'}_2\left(\| y_0\|^2_V + \| f\|^2_{L^2(t_0,t_0+T;H)} \right),
\end{equation}
where ${c'}_2 :={c'}_2(\tilde{N}(a,b),T)$.

Further,  from  \eqref{e39}, \eqref{e40}, and  \eqref{e41}, we can infer that
\begin{equation}
\label{e42}
\begin{split}
\|\partial_t y\|_{L^2(t_0,t_0+T;H)} \leq c \left(\nu+\tilde{N}(a,b)\right) \| y\|_{L^2(t_0,t_0+T;D(A))} + \|f\|_{L^2(t_0,t_0+T;H)}.
\end{split}
\end{equation}
From \eqref{e42} we can extract a subsequence which is also weakly convergent with respect to $H^1(t_0,t_0+T, H)$. Hence \eqref{e43} follows form \eqref{e41}, \eqref{e42}, and the fact that $W(t_0,t_0+T;D(A);H)$ is continuously embedded in the space $C([t_0,t_0+T];V)$.
\end{proof}
In the following  lemma,  an estimate expressing the smoothing property of \eqref{e17}  will be given. This estimate is essential to derive the exponential stability of $RHC$ with respect to the $V$-norm.
\begin{lemma}
\label{lem4}
Let the regularity condition \eqref{e36} be satisfied and  $(T,t_0, y_0, f ) \in   \mathbb{R}^2 \times  H \times L^2(t_0,t_0+T;H)$ be given. Then for the solution $y$ to \eqref{e17} we have the following estimate
\begin{equation}
\label{e45}
\|y(T+t_0)\|^2_{V} \leq c_3(T)\left(\|y_0\|^2_H + \| f\|^2_{L^2(t_0,t_0+T;H)} \right),
\end{equation}
where $c_3(T)=c_3(T,\nu,a,b)>0$.
\end{lemma}
\begin{proof}
First we show that
\begin{equation}
\label{e46}
\sqrt{\cdot-t_0}y\in L^2(t_0,t_0+T;D(A)) \cup L^{\infty}(t_0,t_0+T;V)   \text{ and  }   \sqrt{\cdot-t_0}\partial_t y \in L^2(t_0,t_0+T;H).
\end{equation}
For the verification of \eqref{e46}, we follow a similar argument to the one given  in \cite{MR769654}[Theorem 3.10] which is done  by using Galerkin approximation and  a-priori estimates. Here we limit ourselves to derive the estimates. Multiplying \eqref{e38} with $t-t_0$ for $t \in (t_0,t_0+T)$, using the estimates \eqref{e39}, \eqref{e40}, and Young's inequality, we obtain
\begin{equation*}
\begin{split}
&\frac{d}{2dt}\|\sqrt{t-t_0}\nabla y(t)\|^2_H +\nu\|\sqrt{t-t_0}\Delta y(t)\|^2_H \\
&\leq   \frac{1}{2}\|y(t)\|^2_V  + (t-t_0)\left(|(a(t)y(t),\Delta  y(t))_H| + |(\nabla \cdot (b(t)y(t)),\Delta y(t))_H|+| (f(t),\Delta y(t))_H)|\right)\\
&\leq   \frac{1}{2}\|y(t)\|^2_V  + \tilde{N}(a,b) \|\sqrt{t-t_0}y(t)\|_V \|\sqrt{t-t_0}\Delta y(t)\|_H+ \| \sqrt{t-t_0}f(t)\|_H \|\sqrt{t-t_0}\Delta y(t)\|_H \\
&\leq   \frac{1}{2}\|y(t)\|^2_V  +\frac{\tilde{N}^2(a,b)}{\nu} \|\sqrt{t-t_0}y(t)\|^2_V+  \frac{\nu}{2}\|\sqrt{t-t_0}\Delta y(t)\|^2_H+ \frac{1}{\nu}\| \sqrt{t-t_0}f(t)\|^2_H
\end{split}
\end{equation*}
Integrating from $t_0$ to $t_0+T$, estimate \eqref{e29}, and Gronwall's inequality, we have that
\begin{equation}
\label{e47}
\|\sqrt{\cdot-t_0}y\|^2_{L^{\infty}(t_0,t_0+T;V)} + \nu \|\sqrt{\cdot-t_0}y \|^2_{L^2(t_0,t_0+T;D(A))}  \leq  \tilde{C} \left( \|y_0\|^2_H + \|f\|^2_{L^2(t_0,t_0+T,H)} \right)
\end{equation}
where the  constant $\tilde{C}>0$ depends on $(T,a,b,\nu)$ and the embedding from $L^2(t_0,t_0+T;H)$ to $L^2(t_0,t_0+T;V')$.  In a similar manner as in \eqref{e42}, it can be shown that
\begin{equation*}
\sqrt{\cdot-t_0}\partial_t y =  \sqrt{\cdot-t_0}\left(\nu\Delta y - ay- \nabla \cdot (by) + f\right) \in L^2(t_0,t_0+T;H).
\end{equation*}
and, thus, from \eqref{e46} and \eqref{e47}, we can conclude that  $\sqrt{\cdot-t_0}y  \in C([t_0+\epsilon,t_0+T];V)$  for every $0<\epsilon \leq T$ and
\begin{equation*}
\|\sqrt{T} y(t_0+T) \|_V^2 \leq  \|\sqrt{\cdot-t_0}y \|^2_{L^{\infty}(t_0,t_0+T;V)} \leq \tilde{C} \left( \|y_0\|^2_H + \|f\|^2_{L^2(t_0,t_0+T,H)} \right).
\end{equation*}
Therefore, we are finished with the verification of \eqref{e45}.
\end{proof}
\section{Well-posedness of the Finite-horizon Problems}
In Step 2 of Algorithm \ref{RHA}  repeated solving of finite horizon optimal control problems of the form \ref{PT} is necessary. Here we  investigate  these optimal control problems.
For a set of actuators $\mathcal{U}_{\omega}:=\{\Phi_i : i=1,\dots,N \} \subset H$ and a  triple $(t_0,T,y_0) \in \mathbb{R}_+^2 \times H$  we consider
\begin{align}
\label{e48}
\tag*{$OP^p_T(t_0,y_0)$}
&\min_{\mathbf{u}\in L^2(t_0,t_0+T;\mathbb{R}^N)}J^p_T(\mathbf{u};t_0,y_0):= \int^{t_0+T}_{t_0}(\frac{1}{2}\| y(t)\|^2_{V}+\frac{\beta}{2} |\mathbf{u}(t)|^2_{*})dt  \\ \label{e50}
\text{ s.t } &\begin{cases}
\partial_t y(t)-\nu\Delta y(t) + a(t)y(t)+ \nabla \cdot (b(t)y(t)) =   B_{\mathcal{U}_{\omega}}\mathbf{u}(t)   &\text{ in } (t_0,t_0+T)\times\Omega,\\
y =0   & \text{ on } (t_0,t_0+T)\times \partial \Omega,\\
y(t_0)=y_0 &\text{ on } \Omega,
\end{cases}
\end{align}
 where  $B_{\mathcal{U}_{\omega}}:=\left[ \Phi_1, \dots, \Phi_N \right]$. These problems can be rewritten as
\begin{equation}
\label{e94}
 J^p_{T}(\mathbf{u};t_0,y_0) =  \mathcal{F}(\mathbf{u})+\mathcal{G}(\mathbf{u}),
\end{equation}
where $\mathcal{F}(\mathbf{u}):= \frac{1}{2}\| L_T^{t_0,y_0} \mathbf{u}\|^2_{L^2(t_0,t_0+T;V)}$, with $L_T^{t_0,y_0}$  the solution operator for  \eqref{e50}, and  $\mathcal{G}(\mathbf{u}):=\frac{\beta}{2}\int^{t_0+T}_{t_0}| \mathbf{u}(t)|^2_* $.  By  Proposition \ref{Theo2}, it follows that  $\mathcal{F}$ is well-defined, convex,  and $C^1$.   Moreover, $\mathcal{G}$ is a proper convex function,  and it is  nonsmooth in case $|\cdot|_* = |\cdot|_1$. Hence,  the nonnegative objective function $J^p_{T}(\mathbf{u};t_0,y_0)$ is weakly lower semi-continuous and coercive,  and  existence of a unique minimizer to  \ref{e48} follows from the  direct method in the calculus of variations, see e.g., \cite{MR2361288}.  Uniqueness follows from the strict convexity of $\mathcal{F}$  which  is justified due to the injectivity of $L_T^{t_0,y_0}$.

\begin{proposition}[Verification of \ref{P1}]
\label{prop1}
For every triple $(t_0,T,y_0) \in \mathbb{R}_+^2 \times H$, the finite horizon  problem \ref{e48} admits a unique minimizer.
\end{proposition}
Next, we derive the first-order optimality condition for \ref{e48}. Since $\mathcal{F}$ is smooth and $\dom(\mathcal{G}) =L^2(t_0,t_0+T;\mathbb{R}^N)$, the first-order optimality condition for the minimizer $\mathbf{u}^*$ can be written as
\begin{equation}
\label{e51}
0 \in \partial (\mathcal{F}+\mathcal{G})(\mathbf{u}^*) = \partial \mathcal{F}(\mathbf{u}^*)+ \partial\mathcal{G}(\mathbf{u}^*)=\{ \mathcal{F}'(\mathbf{u}^*)\}+\partial\mathcal{G}(\mathbf{u}^*),
\end{equation}
where $\mathcal{F}'$ is the first Fr\'echet derivative.  We introduce the following adjoint equation,
\begin{equation}
\label{e52}
\begin{cases}
-\partial_tp(t)-\nu \Delta p(t)+ a(t)p(t)-(b(t) \cdot \nabla p(t)) = -\Delta y^*(t)    &\text{ in } (t_0,t_0+T)\times\Omega,\\
p =0   & \text{ on } (t_0,t_0+T)\times \partial \Omega,\\
p(t_0+T)=0 &\text{ on } \Omega,
\end{cases}
\end{equation}
with $y^*(\mathbf{u}^*) \in W(t_0,t_0+T;V,V')$ as the solution of \eqref{e50} for $\mathbf{u}^* \in L^2(t_0,t_0+T;\mathbb{R}^N) $  in the place of $\mathbf{u}$. Then  $\mathcal{F}'(\mathbf{u}^*)$ can be expressed as  $\mathcal{F}'(\mathbf{u}^*) =-B_{\mathcal{U}_{\omega}}^*p$ in  $L^2(t_0,t_0+T: \mathbb{R}^N )$. Therefore,  the optimality condition \eqref{e51} can be stated as
\begin{equation}
\label{e55}
B_{\mathcal{U}_{\omega}}^*p \in \partial \mathcal{G}(\mathbf{u}^*),
\end{equation}
where $B_{\mathcal{U}_{\omega}}^*$ is the adjoint operator to $B_{\mathcal{U}_{\omega}}$ and $p(y^*) \in W(t_0,t_0+T;V,V')$.  Well-posedness of the adjoint equation follows from a similar argument as in the proof of Proposition \ref{Theo2} and the fact that  $ \Delta y^* \in L^2(t_0, t_0+T; V')$.

To deal with the sub-problems of the form \ref{e48} numerically,  we employ a proximal point-type algorithm. These methods are based on the iterative evaluations of the proximal operator
\begin{equation*}
\prox_{\mathcal{G}}(\hat{\mathbf{u}}): L^2(t_0,t_0+T;\mathbb{R}^N)  \to L^2(t_0,t_0+T;\mathbb{R}^N),
\end{equation*}
which is defined by
\begin{equation*}
\prox_{\mathcal{G}}(\hat{\mathbf{u}}) := \argmin_{\mathbf{u} \in L^2(t_0,t_0+T;\mathbb{R}^N)}\left( \frac{1}{2} \|  \mathbf{u} -\hat{\mathbf{u}}\|^2_{ L^2(t_0,t_0+T;\mathbb{R}^N)} +\mathcal{G}(\mathbf{u}) \right).
\end{equation*}
Well-posedness of $\prox_{\mathcal{G} }$ is  justified by the fact that  $\mathcal{G}$ is proper, convex, and weakly lower semi-continuous. We have the following proposition which expresses the first-order optimality conditions in terms of the proximal operator. This optimality condition suggests the termination condition for the proximal point algorithm that we will use later.
\begin{proposition} Let a triple  $(t_0,T,y_0) \in \mathbb{R}_+^2 \times H$ and  $\bar{\alpha} >0$ be given. Then  $\mathbf{u}^* \in L^2(t_0,t_0+T;\mathbb{R}^N)$ is the unique minimizer to \ref{e48} iff there exists a solution $p(y^*(\mathbf{u}^*)) \in W(t_0,t_0+T;V,V')$ to \eqref{e52} such that  the following equality holds
\begin{equation}
\label{e54}
\mathbf{u}^* = \prox_{\bar{\alpha}\mathcal{G}}(\mathbf{u}^* -\bar{\alpha} B_{\mathcal{U}_{\omega}}^*p ),
\end{equation}
where $y^*(\mathbf{u}^*)$ is the solution to \eqref{e50}.
\end{proposition}
\begin{proof}
One needs only to verify the equivalence  between the inequalities \eqref{e55} and \eqref{e54} which is done as in  \cite{MR2798533}[Corollary 26.3].
\end{proof}

\begin{remark} Note that if, in the definition of $\mathcal{G}$, the norm $|\cdot |_{*}$ is chosen to be the $\ell_2$-norm, then $\mathcal{G}$ is smooth and we have $\partial \mathcal{G}(\mathbf{u}) = \{ \mathcal{G}'(\mathbf{u})\}$ and, thus, the optimality conditions \eqref{e55} and \eqref{e54} can be rewritten as
\begin{equation*}
B_{\mathcal{U}_{\omega}}^*p = \beta \mathbf{u}  \quad  \text{  in } L^2(t_0,t_0+T; \mathbb{R}^N).
\end{equation*}
If, on the other hand,  $|\cdot |_{*}$ is chosen to be $\ell_1$-norm, then $\mathcal{G}$ is non-smooth and

\begin{equation*}
 \partial \mathcal{G}(\mathbf{u}) :=  \left\lbrace \mathbf{v} \in L^2(t_0,t_0+T;\mathbb{R}^N)  :  \mathbf{v}(t) \in  \partial ( \frac{\beta}{2}|\mathbf{u}(t)|^2_{1})  \text{ for a.e. } t\in (t_0,t_0+T)  \right\rbrace, 
\end{equation*}
see \cite{MR2798533}[Proposition 16.63].
\end{remark}
In the following, we give the pointwise characterization of  $\prox_{\bar{\alpha}\mathcal{G}}$ for the case $|\cdot|_* = | \cdot|_1$. This characterization is foundational for our optimization algorithm. Due to \cite{MR2798533}[Proposition 24.13] and by setting $g:=\frac{\beta}{2}| \cdot|^2_{1}$, the proximal operator $\prox_{\bar{\alpha}\mathcal{G}}$, can be expressed pointwise as
\begin{equation*}
\left[\prox_{\bar{\alpha}\mathcal{G}}(u)\right] (t) = \prox_{\bar{\alpha} g} (u(t))  \quad  \text{ for almost every } t\in (t_0 ,t_0+T ),
\end{equation*}
and thus  the first-order optimality conditions \eqref{e54} can be stated as
\begin{equation}
\label{e87}
\mathbf{u}^*(t) = \prox_{\bar{\alpha} g}(\mathbf{u}^*(t) -\bar{\alpha} B_{\mathcal{U}_{\omega}}^*p(t)) \quad  \text{ for almost every } t\in (t_0,t_0+T).
\end{equation}
Therefore, its remain only to compute the proximal operator of   $\bar{\alpha} g  : \mathbb{R}^N \to \mathbb{R}_+$.  By following the same argument as in \cite{MR3719240}[Lemma 6.70] and \cite{MR3380641}, it can be shown for every $ \mathbf{x}:=(x_1,\dots,x_N)^t \in \mathbb{R}^N$ that
\begin{equation}
\label{e85}
\prox_{\bar{\alpha} g }(\mathbf{x}) := \begin{cases} \left(\frac{\lambda_ix_i}{\lambda_i+\bar{\alpha} \beta} \right)^N_{i =1},  & \text{ if } \mathbf{x} \neq 0, \\
0, &  \text{ if } \mathbf{x} = 0,

\end{cases}
\end{equation}
where $\lambda_i := \left[ \frac{\sqrt{\frac{\bar{\alpha} \beta}{2}}|x_i| }{\sqrt{\mu^*}}-\bar{\alpha} \beta \right]_+$ with $\mu^*$ being any positive zero of the following one-dimensional nonicreasing function
\begin{equation}
\label{e95}
\psi(\mu): =   \sum^N_{i=1} \left[ \frac{\sqrt{\frac{\bar{\alpha} \beta}{2}}|x_i| }{\sqrt{\mu}}-\bar{\alpha} \beta \right]_+-1,
\end{equation}
where  $\left[ \cdot \right]_+:= \max(0, \cdot )$. In other words,  $\mu^*$  is chosen so that  $\sum^N_{i=1} \lambda_i = 1$.
\begin{remark}
Due to the  characterization \eqref{e85} of $\prox_{\bar{\alpha} g}$,  the cardinality of the set $D(\mathbf{x}) := \{ i\in \{ 1,\dots,N \}  :   \frac{\sqrt{\frac{\bar{\alpha} \beta}{2}}|x_i| }{\sqrt{\mu^*}}-\bar{\alpha} \beta > 0 \}$ is the number of nonzero components  of $\prox_{\bar{\alpha} g}(\mathbf{x})$.  Hence, due to \eqref{e87},   $|D^*(t)|: =  | D(\mathbf{u}^*(t) - B_{\mathcal{U}_{\omega}}^*p(t))| $ stands for the number of non-zero components of  $\mathbf{u}^*$ at time $t$.
\end{remark}

\section{Stabilizability}
In this section,  we summarize selected results on the stabilization of \eqref{e1} by finitely many controllers, in a framework  which is  convenient  for our further discussion. The importance of stabilization by control associated to finitely actuators has been studied in several papers, see e.g.,  \cite{MR2784695, MR2005128,   MR2974739, MR2104285,MR2629581}. Here we follow the same arguments as in \cite{MR3691212,MR3337988,7330942,Phan2018} which deal with time-varying controlled systems.  We will see that under suitable condition on the set of actuators $\mathcal{U}_{\omega}:=\{\Phi_i : i=1,\dots,N \}$, there exists a stabilizing  control $\hat{\mathbf{u}}=\hat{\mathbf{u}}(y_0)$ with $\hat{\mathbf{u}}(t)=(u_1(t),\dots,u_N(t))^t$ which steers the system \eqref{e1} to zero.

Let  $\mathcal{U}_{\omega} \subset H$ be a set of linearly independent functions, and denote by $\Pi_N : H \to \spn{(\mathcal{U}_{\omega})} \subset H$ the orthogonal projection onto $\spn{(\mathcal{U}_{\omega})}$ in  $H$.    We consider the exponential stabilizability of the controlled system
 \begin{equation}
 \label{e25}
 \begin{cases}
\partial_t y(t)-\nu\Delta y(t) + a(t)y(t)+ \nabla \cdot (b(t)y(t)) +\Pi_{N}\mathbf{q}(t)  = 0  &\text{ in } (t_0,\infty)\times\Omega,\\
y=0   & \text{ on } (t_0,\infty)\times \partial \Omega,\\
y(t_0)=y_0 &\text{ on } \Omega,
\end{cases}
\end{equation}
by the control $\mathbf{q} \in L^2(t_0,\infty;H)$. We can express the control term equivalently as
\begin{equation}
\label{e63}
- \Pi_{N} \mathbf{q}(t)  = \sum^N_{i =1} \hat{u}_i(t)\Phi_i,
\end{equation}
where
\begin{equation}
\label{e62}
\mathbf{\hat{u}}=(\hat{u}_1, \dots, \hat{u}_N)^t := -\mathcal{I}\Pi_{N}\mathbf{q},
\end{equation}
and   $\mathcal{I}: \spn{(\mathcal{U}_{\omega})} \to  \mathbb{R}^N$ denotes the canonical isomorphism.

 The following result provides a  sufficient condition on  $\mathcal{U}_{\omega}$ for the exponential stabilizability of \eqref{e25}.

 \begin{proposition}[Exponential stabilizability uniformly with respect to $(t_0,y_0)$ ]
\label{theo1}
Let $\lambda>0$ be given. Then there exists a constant $\Upsilon:=\Upsilon(\lambda,a,b)>0$ such that: if for $\mathcal{U}_{\omega}$ the following condition holds
\begin{equation}
\label{e21}
\tag{coac}
\| I-\Pi_{N}\|^2_{\mathcal{L}(H,V')}  < \Upsilon^{-1},
\end{equation}
then the control system \eqref{e25} is exponentially stabilizable. That is, for every $(t_0,y_0) \in \mathbb{R}_+ \times H$, there exists a control $\mathbf{q}(y_0,\lambda) \in L^2(t_0,\infty;H)$ such that
\begin{align}
\label{e27}
\|y(t)\|^2_H \leq  \Theta_1 e^{-\lambda(t-t_0)} \|y_0\|_H^2  &    \quad  \text{ for  } t\geq t_0,    \\ \label{e28}
\|e^{\bar{\lambda}(\cdot-t_0)}\mathbf{q}\|^2_{L^2(t_0,\infty;H)} \leq {\Theta_2} \|y_0\|^2_H &    \quad  \text{ for  } \bar{\lambda}  < \lambda,
\end{align}
where the constants $\Theta_1$ and $\Theta_2$ depend on $a$, $b$, $\mathcal{U}_{\omega}$ and $\nu$, but are independent  of $y_0$.
 \end{proposition}
The proof can be carried by similar arguments as in \cite{MR3691212}[Theorem 2.10] and \cite{Phan2018}[Introduction].

We observe that if \eqref{e21} holds, then the infinite horizon problem \ref{opinf}  is well-posed.
In this case   the control given in \eqref{e62} satisfies  $\hat{\mathbf{u}} \in L^2(0,\infty;\mathbb{R}^N)$  and  due to  \eqref{e27} and \eqref{e28},
 the infinite horizon performance index function  $J^p_{\infty}(\hat{\mathbf{u}},0, y_0)$ is bounded.  Well-posedness of \ref{opinf} then follows  by the direct method of calculus of variations.

Let us also briefly recall  situations for which  condition \eqref{e21} is satisfied. One such case relates to the choice of  $\{ \Phi \}^N_{i=1}$ as the  eigenfunctions of the negative Laplacian $-\Delta$ with homogeneous Dirichlet boundary conditions, see \cite{MR3691212,MR3337988,Phan2018} for more details. In practice we may be more interested in the choice of actuators given by indicators functions of subsets of  $\omega \subset \Omega$.  To describe one such situation we choose $\omega$ as an open rectangle of the form
\begin{equation}
\label{e100}
\omega_{rect} :=  \prod^n_{i = 1} (l_i,u_i) \subset \Omega.
\end{equation}
We consider the uniform partitioning of $\omega_{rect}$ to a family of sub-rectangles. For every $i=1,\dots,n$ the interval $(l_i,u_i)$ is divided into $d_i$ intervals defined by $I_{i,k}=(l_i+k_i\frac{u_i-l_i}{d_i}, l_i+(k_i+1)\frac{u_i-l_i}{d_i})$ with $k_i\in \{0,1,\dots,d_i-1\}$. Consequently   $\omega_{rect}$ is divided into  $N:=\prod^n_{i =1} d_i$ sub-rectangles defined by
\begin{equation}
\label{e92}
 \{R_i : i \in \{ 1,\dots,N\} \}:=\{ \prod^n_{i =1}I_{i,k_i} : k_i \in \{ 0,1,\dots, d_i  -1\} \},
\end{equation}
and  the set of actuators is defined by
\begin{equation}
\label{e93}
\mathcal{U}_{\omega_{rect}}: = \left\{ \mathbf{1}_{R_i}  : i=1,\dots,N \right\}
\end{equation}
 where  $\mathbf{1}_{R_i}$ is the indicator function of ${R_i}$.
 For this choice of actuators it was shown in \cite{MR3691212}[Example 2.12]  and \cite{7330942}[Section IV]  that \eqref{e21} is satisfied provided that  $N \geq \left( \frac{\bar{I}^2}{\pi^2}\Upsilon \right)^{\frac{n}{2}}$, where  $\bar{I}:= \max_{1\leq i \leq N} (u_i-l_i)$. This relation gives us a lower bound on the number of actuators for which the exponential stabilizabillty is obtained. This bound is defined with respect to the chosen $\lambda>0$, $a$, $b$, $\nu$,  and set of actuators defined by \eqref{e92} and \eqref{e93}, where this dependence is expressed in terms of the value of  $\Upsilon(\lambda,a,b, \mathcal{U}_{\omega_{rect}})$ and $\bar{I}$.

\section{Main Results}
In this section we investigate the  exponential stability of RHC computed by Algorithm \ref{RHA}. First we verify the properties \ref{P2} and \ref{P3} for the value function $V_T: \mathbb{R}_+ \times H \to \mathbb{R}_+$ defined by minimizing the performance index \eqref{e49} subject  to equation \eqref{e50}. Throughout, we use the notation $C_{\mathcal{U}_{\omega}}:=N\max_{ 1 \leq i  \leq N} \| \Phi_i\|^2_H$ for the the set of actuators  $\mathcal{U}_{\omega}:=\{\Phi_i : i=1,\dots,N \}$.
\begin{proposition}
\label{prop2p3}
Let $T\in(0,\infty)$ be given.  Then,  there exists a constant $\gamma_1(T)>0$ depending on $T$ such that \eqref{e8} holds for $V_T$ corresponding to \ref{e48}.  Moreover assume, in addition, that for chosen set of actuators $\mathcal{U}_{\omega} \subset H$ and  $\lambda>0$, condition \eqref{e21} holds with a real number $\Upsilon>0$. Then there exists a nondecreasing, continuous,  and bounded function $\gamma_2: \mathbb{R}_+ \to   \mathbb{R}_+$  such that \eqref{e7} holds for $V_T$ of \ref{e48} with the set of actuators $\mathcal{U}_{\omega}$.  Thus \ref{P2} and \ref{P3} hold.
\end{proposition}
\begin{proof}
Let $(t_0,y_0) \in \mathbb{R}_+ \times H$ be arbitrary. We consider both cases  $|\cdot |_{*}=|\cdot |_{1}$ and $|\cdot |_{*}=|\cdot |_{2}$ simultaneously.   First we deal with \eqref{e8}.
 For any  $\mathbf{u} \in L^2(t_0,t_0+T; \mathbb{R}^N)$, we obtain by \eqref{e14} for \eqref{e50} that
\begin{equation}
\label{e58}
\begin{split}
\|y_0\|^2_H  &\leq  \hat{c}_{\nu}\left(1 + \frac{1}{T} + N(a,b)\right) \int^{t_0+T}_{t_0}\|y(t)\|^2_{V}dt+ \int^{t_0+T}_{t_0}\|B_{\mathcal{U}_{\omega}}\mathbf{u}(t)\|^2_{V'}dt.
\end{split}
\end{equation}
Moreover, with $i_{H,V'}$ the embedding constant from $H$ into $V'$ we estimate
\begin{equation}
\begin{split}
\label{e59}
 \int^{t_0+T}_{t_0}\|B_{\mathcal{U}_{\omega}}\mathbf{u}(t)\|^2_{V'}dt &\leq  i_{H,V'} \int^{t_0+T}_{t_0}\|B_{\mathcal{U}_{\omega}}\mathbf{u}(t)\|^2_{H}dt \leq i_{H,V'}C_{\mathcal{U}_{\omega}} \int^{t_0+T}_{t_0}|\mathbf{u}(t)|^2_{2}\\  &\leq i_{H,V'}C_{\mathcal{U}_{\omega}} \int^{t_0+T}_{t_0}|\mathbf{u}(t)|^2_{1}.
\end{split}
\end{equation}
Defining
\begin{equation}
\label{e88}
\begin{split}
\gamma_1(T) &:= \left( \max\left\{ 2\hat{c}_{\nu}\left(1+\frac{1}{T}+N(a,b)\right),  \frac{2}{\beta}( i_{H,V'}C_{\mathcal{U}_{\omega}}) \right \} \right)^{-1} \\&= \min\left\{  \frac{T}{2\hat{c}_{\nu}(T+1+TN(a,b))}, \frac{\beta}{2 i_{H,V'}C_{\mathcal{U}_{\omega}}} \right \},
\end{split}
\end{equation}
we obtain with  \eqref{e58} and  \eqref{e59} that
\begin{equation*}
\begin{split}
\gamma_1(T)\|y_0\|^2_H   &\leq  \int^{t_0+T}_{t_0} \left(\frac{1}{2}\|y(t)\|^2_{V}+ \frac{\beta}{2}|\mathbf{u}(t)|^2_{*}\right)dt= J^p_T(\mathbf{u};t_0,y_0)
\end{split},
\end{equation*}
for arbitrary $\mathbf{u} \in L^2(t_0,t_0+T; \mathbb{R}^N)$, and  \ref{P3} follows.

Now we turn to the verification of \ref{P2}. Due to Proposition \ref{theo1} applied to \eqref{e25}  there exists a control $\mathbf{q}(y_0,\lambda) \in L^2(t_0,\infty;H)$ such that
\begin{equation}
\label{e60}
\|e^{\frac{\lambda}{2}(\cdot-t_0)}\mathbf{q}\|^2_{L^2(t_0,\infty;H)} \leq {\Theta_2} \|y_0\|^2_H,
\end{equation}
and for it corresponding state,  we have
\begin{equation}
\label{e61}
\|y(t)\|^2_H \leq  \Theta_1 e^{-\lambda(t-t_0)} \|y_0\|_H^2      \quad  \text{ for all } t\geq t_0,
\end{equation}
where the constants $\Theta_1$ and $\Theta_2$ have been defined in Proposition \ref{theo1}. By \eqref{e60} we obtain that
\begin{equation}
\label{e65}
\|\mathbf{q}\|^2_{L^2(t_0,t_0+T;H)} \leq \frac{\Theta_2 }{\lambda}(1-e^{-\lambda T})  \|y_0\|^2_H.
\end{equation}
Defining $\hat{\mathbf{u}}(t)= (\hat{u}_1(t), \dots, \hat{u}_N(t))^t$ and as in \eqref{e62} for $\mathcal{U}_{\omega}$ and using \eqref{e65} , we find
\begin{equation}
\begin{split}
\label{e64}
&\int^{t_0+T}_{t_0} |\hat{\mathbf{u}}(t)|^2_1dt \leq  N\int^{t_0+T}_{t_0}| \hat{\mathbf{u}}(t)|^2_2dt \leq  N\int^{t_0+T}_{t_0}  |\mathcal{I}\Pi_{N} \mathbf{q}(t)|^2_2dt\\
& \leq   N\int^{t_0+T}_{t_0}  \|\mathcal{I}\Pi_{N}\|^2_{\mathcal{L}(H,\mathbb{R}^N)}\| \mathbf{q}(t)\|^2_Hdt \leq Nc_4\int^{t_0+T}_{t_0}\| \mathbf{q}(t)\|^2_H dt \leq   \frac{ N c_4\Theta_2 }{\lambda}(1-e^{-\lambda T}) \|y_0\|^2_H,
\end{split}
\end{equation}
where $c_4$ depends only on $\Omega$ and $\mathcal{U}_{\omega}$. Moreover, by setting $ \mathbf{u} = \hat{\mathbf{u}}$  in equation \eqref{e50}, \eqref{e61} holds. Now using a similar estimate as in \eqref{e30} for \eqref{e50} with $\hat{\mathbf{u}}$ in place of  $ \mathbf{u}$,  we obtain for almost every $t \in (t_0, t_0+T)$ that
\begin{equation}\label{e66}
\begin{split}
\frac{d}{2dt}&\|y(t)\|^2_H + \nu \|y(t)\|^2_V  \\
&\leq   cN(a,b)\|y(t) \|_H\| y(t)\|_V  +  \|B_{\mathcal{U}_{\omega}}\hat{\mathbf{u}}(t)\|_{H}\|y(t)\|_{H}\leq  \frac{1}{2} c_5 \|y(t)\|^2_H + \frac{\nu}{2}\|y(t)\|^2_V  + \frac{1}{2} | \hat{\mathbf{u}}(t)|^2_{2},
\end{split}
\end{equation}
where
$c_5:=\left(\frac{c^2}{\nu}N^2(a,b)+C_{\mathcal{U}_{\omega}}\right).$
Integrating \eqref{e66} over $(t_0,t_0+T)$ and using \eqref{e61} and \eqref{e64}, we obtain
\begin{equation}
\begin{split}
\label{e67}
\int^{t_0+T}_{t_0}& \|y(t)\|^2_V\,dt \leq \frac{1}{\nu}\left(\|y(t_0)\|^2_H +  c_5\int^{t_0+T}_{t_0} \|y(t)\|^2_H\,dt + \int^{t_0+T}_{t_0}| \hat{\mathbf{u}}(t)|^2_{2}\,dt \right)\\
&\leq \frac{1}{\nu}\left(1 +\frac{c_5\Theta_1+c_4\Theta_2  }{\lambda}(1-e^{-\lambda T})  \right)\|y_0\|^2_H.
\end{split}
\end{equation}
Now we can show \eqref{e7}.  Due to the  optimality of  $V_T$ and using \eqref{e64} and \eqref{e67}, we have for the case $|\cdot|_1$ that
\begin{equation*}
\begin{split}
V_T(t_0,y_0) &\leq J^p_T(\hat{\mathbf{u}};t_0,y_0)=\int^{t_0+T}_{t_0}(\frac{1}{2}\| y(t)\|^2_{V}+\frac{\beta}{2} |\hat{\mathbf{u}}(t)|^2_{1})dt \\
&\leq  \frac{1}{2\nu}\left(1 +\frac{c_5\Theta_1+(1+\nu N\beta)c_4\Theta_2  }{\lambda}(1-e^{-\lambda T})\right) \|y_0\|^2_H =: \gamma^{\ell_1}_2(T) \|y_0\|^2_H,
\end{split}
\end{equation*}
and, thus, \eqref{e7} holds for the choice  $\gamma_2(T) =\gamma^{\ell_1}_2(T)$.
In a similar manner it can be shown that \eqref{e7} holds for the case  $|\cdot|_{*}= |\cdot|_2$ with the choice of $\gamma_2(T) =\gamma^{\ell_2}_2(T) $ defined by
\begin{equation*}
\gamma^{\ell_2}_2(T) := \frac{1}{2\nu}\left(1 +\frac{c_5\Theta_1+(1+\nu\beta)c_4\Theta_2  }{\lambda}(1-e^{-\lambda T}) \right),
\end{equation*}
and thus we are finished with the verification of \eqref{e7}.
\end{proof}
In the next theorem, we prove the exponential stability of RHC obtained by Algorithm \ref{RHA}. Moreover, it will be shown that, for more regular data, we obtain a stronger stability result.
\begin{theorem}
\label{Theo5}
Assume that for given $\mathcal{U}_{\omega}\subset H$ and $\lambda>0$,  condition \eqref{e21} is satisfied with a real number $\Upsilon>0$.   Then, for given $\delta$,  there exist numbers $T^*(\delta,\mathcal{U}_{\omega})>\delta$ and $\alpha(\delta,\mathcal{U}_{\omega})$ such that for every prediction horizon $T\geq T^*$, the receding horizon control $\mathbf{u}_{rh} \in L^2(0,\infty;\mathbb{R}^N)$ obtained by Algorithm \ref{RHA} is globally \textbf{suboptimal}  and \textbf{exponentially stable}, i.e.
inequalities \eqref{ed27} and   \eqref{ed28} hold for every $y_0\in H$. If additionally \eqref{e36} holds,  then we obtain that
\begin{equation}
\label{e68}
\|y_{rh}(t)\|^2_{V} \leq c_Ve^{-\zeta t}\|y_0\|^2_{V} \quad  \text{ for }  t\geq 0 ,
\end{equation}
for every  $y_0 \in V$,   where  $\zeta$ has been  defined  in Theorem \ref{subopth}  and  $c_V$  depends on $\alpha(\delta,\mathcal{U}_{\omega})$, $\delta$, and $T$, but is independent of $y_0$.
\end{theorem}

\begin{proof}
Clearly, we need only to verify the assumptions of Theorem \ref{subopth}. Well-posedness,  and  justification of  estimate \eqref{Est1} for equation  \eqref{e50} follows from Proposition \ref{Theo2} and estimate \eqref{e29}. Further, due to the Poincar\'e inequality, condition \eqref{estiob} holds for the incremental function defined by
\begin{equation}
\label{e74}
\ell(t,y,\mathbf{u}):= \frac{1}{2}\|\nabla y\|^2_{H}+\frac{\beta}{2} |\mathbf{u}|^2_{*}.
\end{equation}
Moreover, due to Propositions \ref{prop1} and \ref{prop2p3}, Properties \ref{P1}, \ref{P2}, and \ref{P3} hold and we are in the position that we can apply Theorem \ref{subopth}. Hence, we can conclude there exist numbers $T^*(\delta,\mathcal{U}_{\omega})>\delta$ and $\alpha(\delta,\mathcal{U}_{\omega})$ such that for every prediction horizon $T\geq T^*$, RHC obtained by Algorithm \ref{RHA} is suboptimal and  exponentially stable.

 Now we turn to the verification  of \eqref{e68}.  For any  $T\geq T^*$, $k\geq 1$, and $y_0 \in V$,  $y_{rh} \in W(t_k,t_{k+1};D(A),H)$ is the solution of the following equations
\begin{equation}
\label{e69}
\begin{cases}
\partial_t y(t)-\nu\Delta y(t) + a(t)y(t)+ \nabla \cdot (b(t)y(t)) =  B_{{\mathcal{U}}_{\omega}}\mathbf{u}_{rh}(t)   &\text{ in } (t_k,t_{k+1})\times\Omega,\\
y =0   & \text{ on } (t_k,t_{k+1})\times \partial \Omega,\\
y(t_k)=y_{rh}(t_k) &\text{ on } \Omega,
\end{cases}
\end{equation}
where   $y_{rh}(t_k) \in V \subset H$.   Using Lemma \ref{lem4} and  estimate \eqref{e45} for \eqref{e69}, we obtain
\begin{equation}
\label{e71}
\begin{split}
\|y_{rh}(t_{k+1})\|^2_{V} &\leq c_3(\delta)\left(\|y_{rh}(t_{k})\|^2_{H} +\int^{t_{k+1}}_{t_{k}} \|B_{\mathcal{U}_{\omega}}\mathbf{u}_{rh}(t)\|^2\,dt\right), \\
&\leq c_3(\delta)\left(\|y_{rh}(t_{k})\|^2_{H} +C_{\mathcal{U}_{\omega}}\int^{t_{k+1}}_{t_{k}} |\mathbf{u}_{rh}(t)|^2_{*}\,dt\right).  
\end{split}
 \end{equation}
Moreover by using  Property \ref{P2}, we obtain 
\begin{equation}
\label{e72}
\int^{t_{k+1}}_{t_{k}} |\mathbf{u}_{rh}(t)|^2_{*}\,dt \leq  \frac{2}{\beta} V_T(t_k,y_{rh}(t_k)) \leq \frac{2\gamma_2(T)}{\beta}\|y_{rh}(t_k)\|^2_{H}.
\end{equation}
Using \eqref{e71}, \eqref{e72} and inequality \eqref{ed32} in the proof of Theorem \ref{subopth} in Appendix \ref{apend1}, we can write that 
\begin{equation*}
\begin{split}
\|y_{rh}(t_{k+1})&\|^2_{V} \leq c_3(\delta)\left(1+  \frac{ 2\gamma_2(T)C_{\mathcal{U}_{\omega}}}{\beta} \right)\|y_{rh}(t_k)\|^2_{H}\leq  c_3(\delta)\left(1+  \frac{ 2\gamma_2(T)C_{\mathcal{U}_{\omega}}}{\beta} \right) c'_He^{-\zeta k\delta}\|y_0\|^2_{H}   \\
&\leq  c_3(\delta)\left(1+  \frac{ 2\gamma_2(T)C_{\mathcal{U}_{\omega}}}{\beta} \right) \left( 1- \frac{\alpha \gamma_1(\delta)}{\gamma_2(T)} \right)^{-1}c'_H i_{V,H} e^{-\zeta( k+1)\delta}\|y_0\|^2_{V}, 
\end{split}
\end{equation*}
 where   $c'_H= \frac{\gamma_2(T)}{\gamma_1(T)}$ and $\zeta$ satisfying $e^{-\zeta \delta } = \left( 1- \frac{\alpha \gamma_1(\delta)}{\gamma_2(T)} \right)$ have been defined in the proof of Theorem \ref{subopth},  and $i_{V,H}$ stands for the continuous embedding from $V$ to $H$.   Therefore, defining
 \begin{equation*}
 c'_V := c'_H i_{V,H}c_3(\delta)\left(1+  \frac{ 2\gamma_2(T)C_{\mathcal{U}_{\omega}}}{\beta} \right) \left( 1- \frac{\alpha \gamma_1(\delta)}{\gamma_2(T)} \right)^{-1},
 \end{equation*}
 we obtain 
 \begin{equation}
\label{e73}
\|y_{rh}(t_{k+1})\|^2_{V}  \leq c'_V e^{-\zeta t_{k+1}}\|y_0\|^2_{V}.
\end{equation}
Moreover,  for every $t>0$, there exist a $k \in \mathbb{N}_0$ such that $t\in [t_k,t_{k+1}]$.  Using \eqref{e43} for  equation \eqref{e69},   \eqref{e72},  and \eqref{e73}, we have
\begin{equation*}
\begin{split}
\|y_{rh}(t)\|^2_{V}&\stackrel{\text{\eqref{e43}}}{\leq} c_2(\delta)\left(\| y_{rh}(t_{k})\|^2_V + \int^{t_{k+1}}_{t_k}\| B_{\mathcal{U}_{\omega}}\mathbf{u}(t)\|^2_{H}\,dt \right) 
\stackrel{\text{\eqref{e72}}}{\leq} c_2(\delta)\left(1+\frac{2 i_{V,H} \gamma_2(T) C_{\mathcal{U}_{\omega}}}{\beta}\right)\|y_{rh}(t_{k})\|^2_{V}\\
                                                             &\stackrel{\text{\eqref{e73}}}{\leq} c_2(\delta)\left(1+\frac{2 i_{V,H} \gamma_2(T) C_{\mathcal{U}_{\omega}}}{\beta}\right) c'_V e^{-\zeta t_{k}}\|y_0\|^2_{V}\\
                                                             &= c_2(\delta)c'_V\left(1+\frac{2 i_{V,H} \gamma_2(T) C_{\mathcal{U}_{\omega}}}{\beta}\right)\left( 1- \frac{\alpha \gamma_1(\delta)}{\gamma_2(T)} \right)^{-1}e^{-\zeta t_{k+1}}\|y_0\|^2_{V}\\
                                                              &\leq   c_2(\delta)c'_V\left(1+\frac{2 i_{V,H} \gamma_2(T) C_{\mathcal{U}_{\omega}}}{\beta}\right)\left( 1- \frac{\alpha \gamma_1(\delta)}{\gamma_2(T)} \right)^{-1}e^{-\zeta t}\|y_0\|^2_{V},
\end{split}
\end{equation*}
and by setting 
\begin{equation}
\label{e90}
c_V:= c_2(\delta)c'_V\left(1+\frac{2 i_{V,H} \gamma_2(T) C_{\mathcal{U}_{\omega}}}{\beta}\right) \left( 1- \frac{\alpha \gamma_1(\delta)}{\gamma_2(T)} \right)^{-1},
\end{equation}
we have \eqref{e68} and the proof is finished.
\end{proof} 

\begin{remark}
{\em Since we have now an estimate for $\gamma_1(\cdot)$ in \eqref{e88},  it is of interest to study the effect of the constant $\delta$ on the constants $\alpha(\delta, T)$, $c_H(\delta,T)$, $c_V(\delta,T)$  given in Theorem \ref{Theo5} for a fixed $T\geq T^*$. Due to \eqref{e3}, as $\delta$ is getting smaller for $ \delta <  \frac{T}{2} $,  $\alpha(\delta, T)$  becomes smaller. Moreover, by definitions of $\gamma_1(\cdot)$ given in \eqref{e88}, we can infer that $\gamma_1(\cdot)$ is an increasing function and it vanishes as $\delta \to 0$. Therefore, by reducing the value of $\delta$ (for $\delta<\frac{T}{2}$) as long as $\frac{\gamma_2^2(T)}{\alpha^2_{\ell}\delta(T-\delta)} \leq 1$ holds,  the value of the factor $\left( 1- \frac{\alpha \gamma_1(\delta)}{\gamma_2(T)} \right)$   in  \eqref{ed313}  is getting larger. On the other hand, due to \eqref{e89} and \eqref{e90}, the transient constants  $c_H(\delta,T) $and $c_V(\delta,T)$ are getting smaller since the constant  $c_{\delta}$, $c_2(\delta)$, and $c_3(\delta)$ are strictly increasing functions.}
 \end{remark}

From numerical and theoretical points of view, it is also of interest  to consider the following incremental function within the receding horizon algorithm \ref{RHA}  
\begin{equation}
\label{e75}
\ell(t,y,u):=\frac{1}{2}\| y\|^2_H+\frac{\beta}{2}| \mathbf{u}|^2_{*}, 
\end{equation}
instead of \eqref{e74}. To be more precise, we want to  penalize  the  $L^2(\Omega)$-tracking term instead of the  $H^1_0(\Omega)$-tracking term.  In this case, we will see that, RHC obtained by Algorithm \ref{RHA}, is suboptimal and asymptomatically stable with respect to the $L^2(\Omega)$-norm. In order to derive,  the exponential stability of RHC, we used Property \ref{P3} (see the second part of Theorem \ref{subopth}). This property does not hold for the value function $V_T$  associated to $\eqref{e75}$.   Indeed,  Property \ref{P3} is directly related to the observability inequality \eqref{e14}, which is not satisfied if we change  $H^1_0(\Omega)$-norm with $L^2(\Omega)$-norm, see  \cite{MR2322815, MR2383077,MR1750109}. Therefore, for proving the asymptotic stability of RHC, we need to use a different technique.

\begin{theorem}
Suppose that for given $\mathcal{U}_{\omega}\subset H$ and $\lambda>0$,  condition \eqref{e21} is satisfied with a real number $\Upsilon>0$. Then, for given $\delta>0$,  there exist numbers $T^*(\delta,\mathcal{U}_{\omega})>\delta$ and $\alpha(\delta,\mathcal{U}_{\omega})$ such that for every prediction horizon $T\geq T^*$, the receding horizon control $\mathbf{u}_{rh} \in L^2(0,\infty;\mathbb{R}^N)$, obtained by Algorithm \ref{RHA} with the incremental function \eqref{e75},  is globally \textbf{suboptimal}  and \textbf{asymptotically stable} with respect to $H$.
\end{theorem}
\begin{proof}
First, we need to verify  Properties \ref{P1} and \ref{P2}. Property \ref{P1} is clearly satisfied since the optimal control problems \ref{e48} with  incremental functions of the form \eqref{e75} are positive, coercive, and weakly sequentially lower-semicontinuous. Moreover, in a similar manner as in the proof of Proposition \ref{prop2p3}  and by using \eqref{e61} and \eqref{e64}, it can be shown that Property \ref{P2} holds for the following choices of $\gamma_2(T)$  depending on the norm $| \cdot |_{*}$
\begin{equation}
\label{e82}
\begin{split}
\gamma^{\ell_1}_2(T) :=\frac{\Theta_1+\beta N c_4\Theta_2 }{2\lambda}(1-e^{-\lambda T}) \quad \text{ and }  \quad \gamma^{\ell_2}_2(T) :=\frac{\Theta_1+\beta  c_4\Theta_2 }{2\lambda}(1-e^{-\lambda T}),
\end{split}
\end{equation}
where the constants  $\Theta_1$, $\Theta_2$, and $c_4$ have been given in the proof of Proposition \ref{prop2p3}.

Now since Properties \ref{P1} and \ref{P2} hold, we are in the position that we can use the first part of  Theorem \ref{subopth}.  Hence,  there exist numbers $T^* > \delta$, and $\alpha \in (0,1)$,  such that for every fixed  prediction horizon $T \geq T^*$ and every  $y_0 \in H$,  the suboptimality inequality \eqref{ed27} holds.

 Now we show that RHC is asymptotically stable i.e., $\lim_{t \to \infty} \|y_{rh}(t)\|_H = 0$ for every $y_0 \in H$. Using the suboptimality inequality \eqref{ed27} and \ref{P2}, we can write
\begin{equation}
\label{e79}
\int^{\infty}_0\|y_{rh}(t)\|^2_H\,dt  \leq \frac{2\gamma_2(T)}{\alpha} \|y_0\|^2_H   \quad  \text{ and }    \quad  \int^{\infty}_0 |\mathbf{u}_{rh}(t)|^2_{*}\,dt  \leq \frac{2\gamma_2(T)}{\alpha \beta} \|y_0\|^2_H.
\end{equation}
Moreover, in a similar manner as in the proof of Proposition \ref{prop2p3} (see \eqref{e66}-\eqref{e67}), it can be shown for every $t\geq t_0$ that
\begin{equation}
\begin{split}
\label{e76}
\|y_{rh}(t) \|^2_H   +  \nu \int^{t}_{0} \|y_{rh}(t)\|^2_V\,dt  &\leq \|y_0\|^2_H +  c_5\int^{\infty}_{0} \|y_{rh}(t)\|^2_H\,dt + \int^{\infty}_{0}| \mathbf{u}_{rh}(t)|^2_{2}\,dt \leq c_6\|y_0\|^2_H,
\end{split}
\end{equation}
where  $c_6 :=\left(1 +\frac{ 2(1+\beta c_5)\gamma_2(T)}{\alpha\beta}\right)$,   and   $c_5$ has been defined in the proof of Proposition \ref{prop2p3}. Due to \eqref{e76}  we can conclude
\begin{equation}
\label{e78}
\|y_{rh}\|^2_{L^{\infty}(0,\infty;H)} \leq  c_6 \|y_0\|^2_H    \quad  \text{ and }    \quad   \int^{\infty}_{0} \|y_{rh}(t)\|^2_V \leq \frac{c_6}{\nu} \|y_0\|^2_H.
\end{equation}
Further, we have for every $t'' \geq t' \geq 0$ that
{\small
\begin{equation}
\begin{split}
\label{e77}
&\|y_{rh}(t'')\|^2_{H}-\|y_{rh}(t')\|^2_{H}=\int^{t''}_{t'}\frac{d}{dt}\|y_{rh}(t)\|^2_{H}dt,\\
&=2\int^{t''}_{t'}\langle  y_{rh}(t), \nu\Delta y_{rh}(t) - a(t)y_{rh}(t)-\nabla \cdot (b(t)y_{rh}(t))+ B_{\mathcal{U}_{\omega}}\mathbf{u}_{rh}(t) \rangle_{V,V'}dt,\\
&=-2\nu\int^{t''}_{t'}\| \nabla y_{rh}(t)\|^2_{H}dt+2\int^{t''}_{t'}\langle - a(t)y_{rh}(t)-\nabla \cdot (b(t)y_{rh}(t))+ B_{\mathcal{U}_{\omega}}\mathbf{u}_{rh}(t),y_{rh}(t)\rangle_{V,V'}dt,\\
&\leq 2N(a,b)\int^{t''}_{t'}\|y_{rh}(t)\|_{V}\|y_{rh}(t)\|_{H}dt+2C_{\mathcal{U}_{\omega}}\int^{t''}_{t'}|\mathbf{u}_{rh}(t)|_{2}\|y_{rh}(t)\|_{H}dt\\
&\leq 2N(a,b)\big( \int^{t''}_{t'}\|y_{rh}(t)\|^2_{V}dt\big)^{\frac{1}{2}}\big( \int^{t''}_{t'}\|y_{rh}(t)\|^2_{H}dt\big)^{\frac{1}{2}}+2C_{\mathcal{U}_{\omega}}\big( \int^{t''}_{t'}|\mathbf{u}_{rh}(t)|^2_{2}dt\big)^{\frac{1}{2}}\big( \int^{t''}_{t'}\|y_{rh}(t)\|^2_{H}dt\big)^{\frac{1}{2}},\\
&\leq   c_7 \|y_0\|^2_{H}(t''-t')^{\frac{1}{2}},
\end{split}
\end{equation}}
where  $c_7:= 2\left(N(a,b){\nu}^{\frac{-1}{2}}c_{6}  +  C_{\mathcal{U}_{\omega}}c^{\frac{1}{2}}_6 \left(\frac{2\gamma_2(T)}{\alpha \beta}\right)^{\frac{1}{2}} \right)$ and in  the last inequality \eqref{e79}  and  \eqref{e78} have been used.  Moreover, due to the left inequality in \eqref{e79}, we infer for any  $L>0$ that
\begin{equation}
\label{e80}
\lim_{t \to \infty}\int^{t}_{t-L}\|y_{rh}(s)\|^2_{H}ds = 0
\end{equation}
 Now suppose to contrary that
\begin{equation*}
\lim_{t \to \infty}\|y_{rh}(t)\|^2_{H} \neq 0.
\end{equation*}
Then there exists an  $\epsilon>0$ and a positive sequence $\{t_n\}^{\infty}_{n=1}$ with $\lim_{n \to \infty}t_n = \infty$ for which
\begin{equation}
\label{e81}
\|y_{rh}(t_n)\|^2_{H}>\epsilon \quad \text{ for all } n = 1,2,\dots.
\end{equation}
It follows from \eqref{e77} and \eqref{e81} that for every $L >0$ and $n = 1,2,\dots$
\begin{equation*}
\begin{split}
\int^{t_n}_{t_n-L} \|y_{rh}(t)\|^2_{H}dt &=\int^{t_n}_{t_n-L} \|y_{rh}(t_n)\|^2_{H}dt - \int^{t_n}_{t_n-L}\big(\|y_{rh}(t_n)\|^2_{H}-\|y_{rh}(t)\|^2_{H}\big)dt,\\
&> L\epsilon-c_7\|y_0\|^2_{H}\int^{t_n}_{t_n-L}(t_n-t)^\frac{1}{2}dt= L\epsilon-\frac{2}{3}c_7\|y_0\|^2_{H}L^{\frac{3}{2}}.
\end{split}
\end{equation*}
Setting $\sigma := \frac{2}{3}c_7 \|y_0\|^2_{H}$, and choosing $L:=(\frac{\epsilon}{2\sigma})^2$, we obtain
\begin{equation*}
\int^{t_n}_{t_n-L} \|y_{rh}(t)\|^2_{H}dt > \frac{L\epsilon}{2} \quad \text{ for all } n=1,2,\dots,
\end{equation*}
 and  this leads to a contradiction  to \eqref{e80}. Hence $\lim_{t \to \infty}\|y_{rh}(t)\|^2_{H}=0$ and the proof is complete.
\end{proof}
\begin{remark}
 \label{Re1}
{\em By comparing   $\gamma^{\ell_1}_2(T)$ with $\gamma^{\ell_2}_2(T)$, from the proof of Proposition \ref{prop2p3} and  \eqref{e82}, we can see that for both of the choices of  \eqref{e74} and \eqref{e75} for the incremental function $\ell$, we have
\begin{equation}
\label{e83}
\gamma^{\ell_1}_2(T) = \gamma^{\ell_2}_2(T)+ \frac{\beta  c_4(N-1)\Theta_2 }{2\lambda}(1-e^{-\lambda T}).
\end{equation}
Hence, on the basis of  \eqref{e3},  we can conclude that $T^*_1 \geq T^*_2$ for every fixed $\delta$ and $\alpha$, where $T^*_1$ and $T^*_2$  correspond to $\ell_1$-norm and $\ell_2$-norm, respectively. Furthermore, it can be seen that for the same value of $\alpha$ we obtain
\begin{equation*}
T^*_1 -T^*_2 \approx O((N-1)^2).
\end{equation*}}
\end{remark}

\section{Numerical Experiments}
In this section we report on our numerical experiments with Algorithm \ref{RHA} for an exponentially unstable parabolic equation which illustrate the theoretical findings. 
Both the $\ell_1$-  and  $\ell_2$-norm for the control penalty terms $\mathcal{G}(\cdot)$ are used and  different values of the prediction horizon $T$ for the fixed value of the sampling time $\delta:=0.25$ are considered. 
Throughout,  we set $T_{\infty}=10$ as  the final computation time and our control domain $\omega$ was defined as an union of two open rectangles of the form \eqref{e100}.  For each of these rectangle, the set actuators were chosen as in \eqref{e93}. The spatial discretization was done by a conforming linear finite element  scheme  using  continuous  piecewise  linear  basis  functions  over  a  uniform triangulation.  The spatial domain was chosen to be $\Omega :=(0,1)^2 \subset \mathbb{R}^2$ and it was discretized by $1089$ cells. Then the ordinary differential equations resulting after spatial discretization were numerically solved by the Crank-Nicolson time stepping  method  with  step-size $\Delta t =0.0125$.  For solving the finite horizon optimal control problems for the  $\ell_2$-norm, we employed the Barzilai-Borwein (BB) gradient method  \cite{azmi2018analysis,MR967848} to the reduced problem \eqref{e94},  where  $\mathcal{G}(\mathbf{u}):=\frac{\beta}{2}\int^{t_k+T}_{t_k}| \mathbf{u}(t)|^2_2$. For this case the BB method was terminated as the $L^2(t_k,t_{k}+T ; \mathbb{R}^N)$-norm of reduced gradient was less than $10^{-5}$. Further, for the case of  $\ell_1$-norm i.e.,  $\mathcal{G}(\mathbf{u}):=\frac{\beta}{2}\int^{t_k+T}_{t_k}| \mathbf{u}(t)|^2_1$, we applied a similar proximal gradient method as that investigated in \cite{MR2792408,MR2678081,MR2650165}  on  problem \eqref{e94}. More precisely,  we followed the iteration rule
\begin{equation*}
\mathbf{u}^{j+1} =\prox_{\alpha_j\mathcal{G}}(\mathbf{u}^j -\alpha_j\mathcal{F}'(u^j) )= \prox_{\alpha_j\mathcal{G}}(\mathbf{u}^j -\alpha_j B_{\mathcal{U}_{\omega}}^*p^{j} ),
\end{equation*}
where $p^{j}:=p(y^{j})$ is the solution of \eqref{e52} for the forcing function $\Delta y^j$ instead of $\Delta y^*$, and  $y^{j}=y(\mathbf{u}^{j})$ is defined as the solution of \eqref{e50} for the control $\mathbf{u}^{j}$ instead of $\mathbf{u}$. Moreover, the stepsize $\alpha_j$ is computed by a non-monotone linesearch algorithm which uses the BB-stepsize corresponding to the smooth part $ \mathcal{F}$  as the initial trial stepsize, see \cite{MR2792408,MR2678081,MR2650165}  for more details.   In this case the optimization algorithm was terminated as the following condition held
\begin{equation*}
\frac{\| \mathbf{u}^{j+1}-\mathbf{u}^{j} \|_{L^2(t_k,t_{k}+T ; \mathbb{R}^N)}  }{\| \mathbf{u}^{j+1} \|_{L^2(t_k,t_{k}+T ; \mathbb{R}^N)}} \leq 10^{-4}.
\end{equation*}
The evaluation of the  proximal operator $\prox_{\bar{\alpha}\mathcal{G}}$ was carried out by pointwise evaluation \eqref{e87} at time grid points. Further, at every time grid point,  $\prox_{\bar{\alpha} g }$ was computed by \eqref{e85}, where the zero $\mu^*$ of the function  $\psi(\mu)$  defined in \eqref{e95} was found by the bisection method with the tolerance $10^{-10}$.
For all numerical tests,  we set $\nu =0.1$,  and defined
\begin{equation*}
 a(t,x):=-2.8-0.8|\sin(t+x_1)|, \quad b(t,x):= \binom{-0.01(x_1+x_2)}{0.2x_1x_2\cos(t)},
\end{equation*}
and $y_0(x):=3\sin(\pi x_1)\sin(\pi x_2)$. For this choice, the uncontrolled state  $y^{un}$ is exponentially unstable. This fact is illustrated in Figures \ref{Fig3} and \ref{Fig4}. The first curve with the black color in Figure \ref{Fig3}  (resp. Figure \ref{Fig4}) is corresponding to the evolution of $\log(\|y^{un}(t)\|_H)$ (resp. $\log(\|y^{un}(t)\|_V)$). Moreover, we have
\begin{equation*}
\|y^{un}\|_{L^2(0,T_{\infty}:V)}=3.20 \times 10^6, \quad  \|y^{un}(T_{\infty})\|_{V}=5.32\times 10^6 , \text{ and } \|y^{un}(T_{\infty})\|_{H}=1.19\times 10^6.
\end{equation*}
Figure \ref{Figa} depicts some snapshots of  the uncontrolled state $y^{un}$.

As performance criteria, we considered the quantities: 1. {\small $J^p_{T_{\infty}}(u_{rh},y_0)$},  2. {\small $\|y_{rh}\|_{L^2(0,T_{\infty};V)}$},  3.{\small $\|y_{rh}(T_{\infty})\|_{V}$}, 4. {\small $\|y_{rh}(T_{\infty})\|_{H}$}, 5. iter : the \textit{total} number of iterations that the optimization algorithm needs for all open-loop problems on the intervals $(t_i, t_{i}+T)$ for $i =0,\dots, r-1$ with $r: = \frac{T_{\infty}}{\delta}$.  All
computations were done in the MATLAB platform.
\begin{figure}[htbp]
    \centering
    \subfigure[$t=0$]
    {
       \includegraphics[height=3cm,width=4cm]{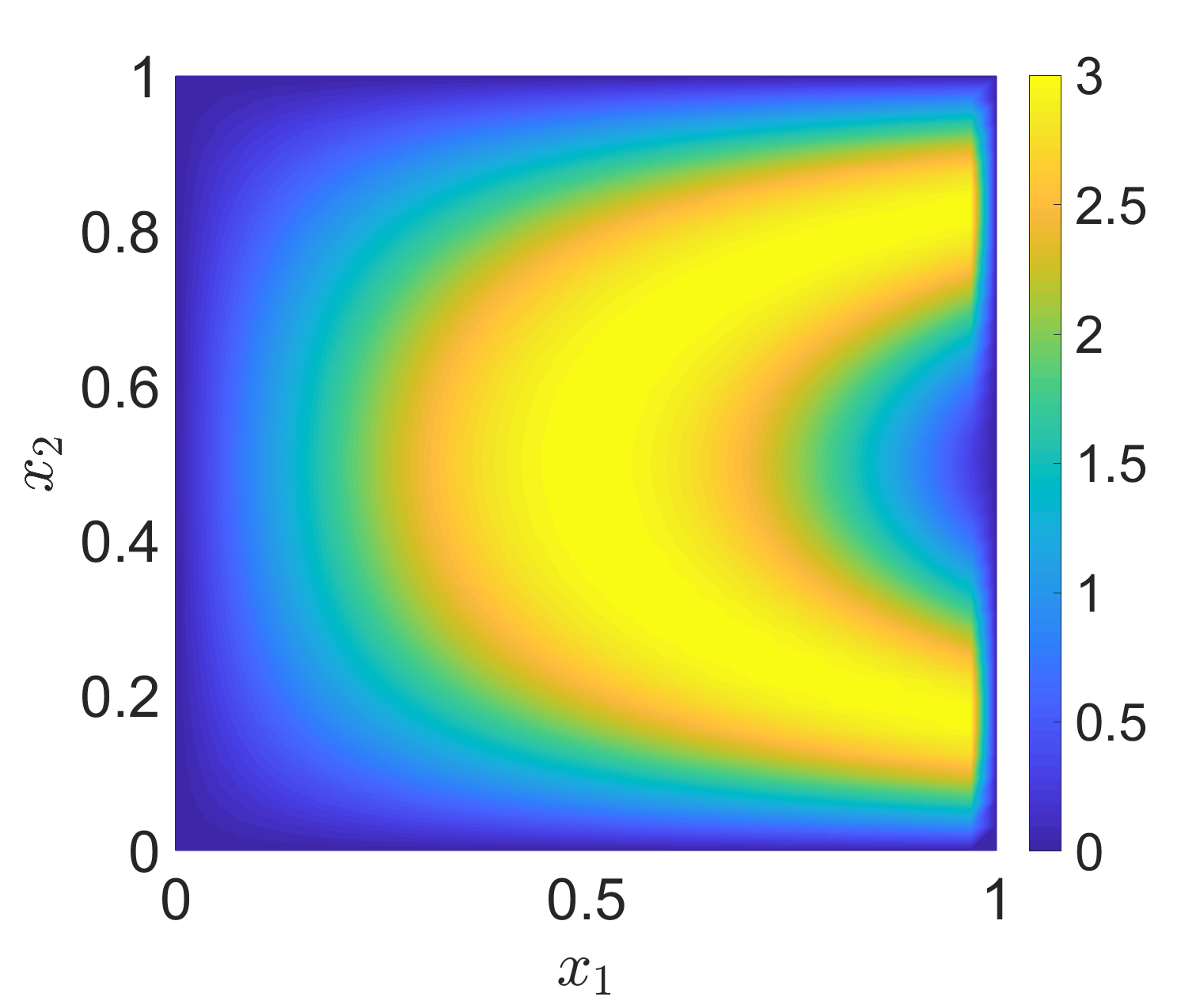}
    }
    \subfigure[ $t =5$]
    {
        \includegraphics[height=3cm,width=4cm]{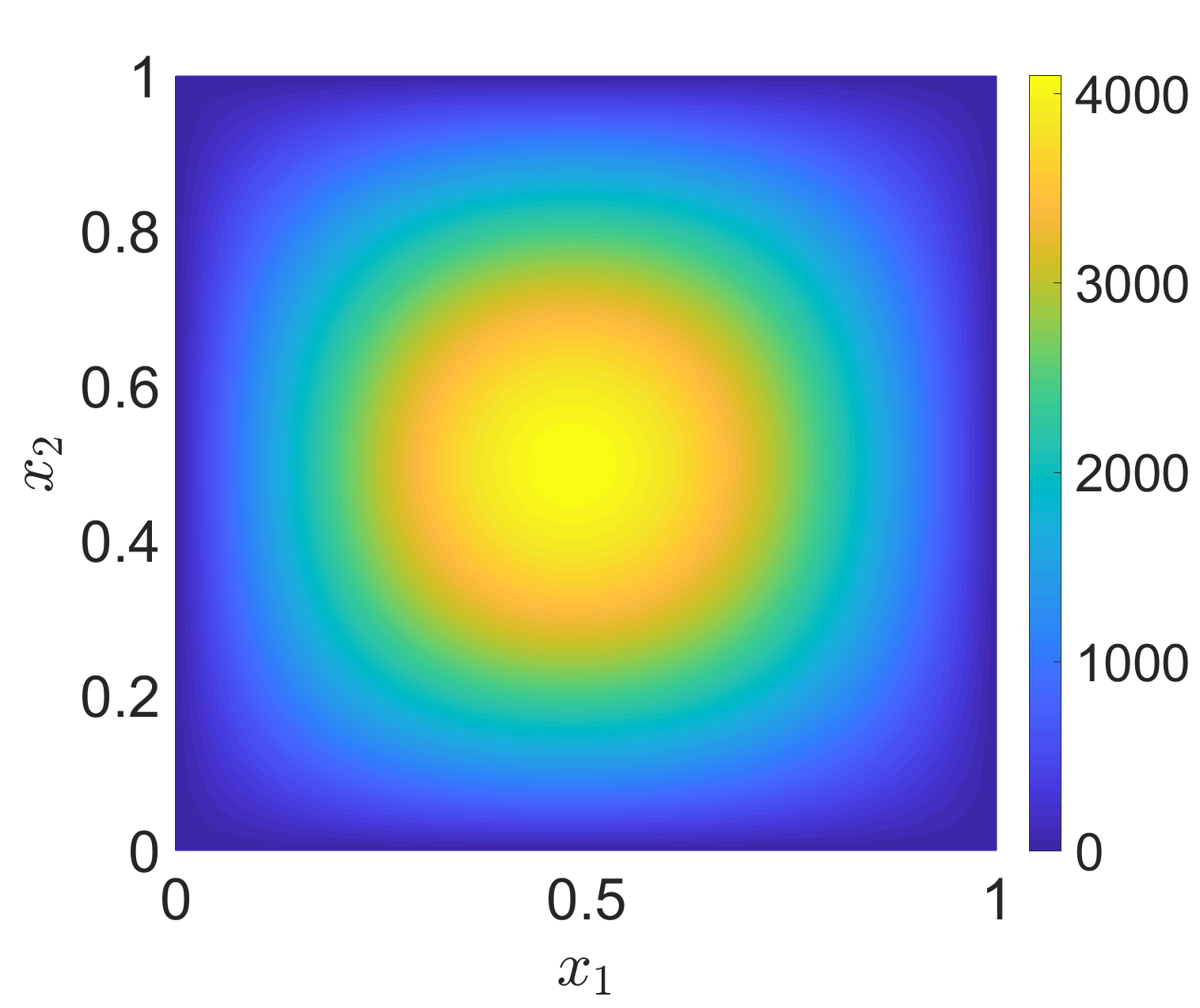}
    }
    \subfigure[ $t=10$]
    {
        \includegraphics[height=3cm,width=4cm]{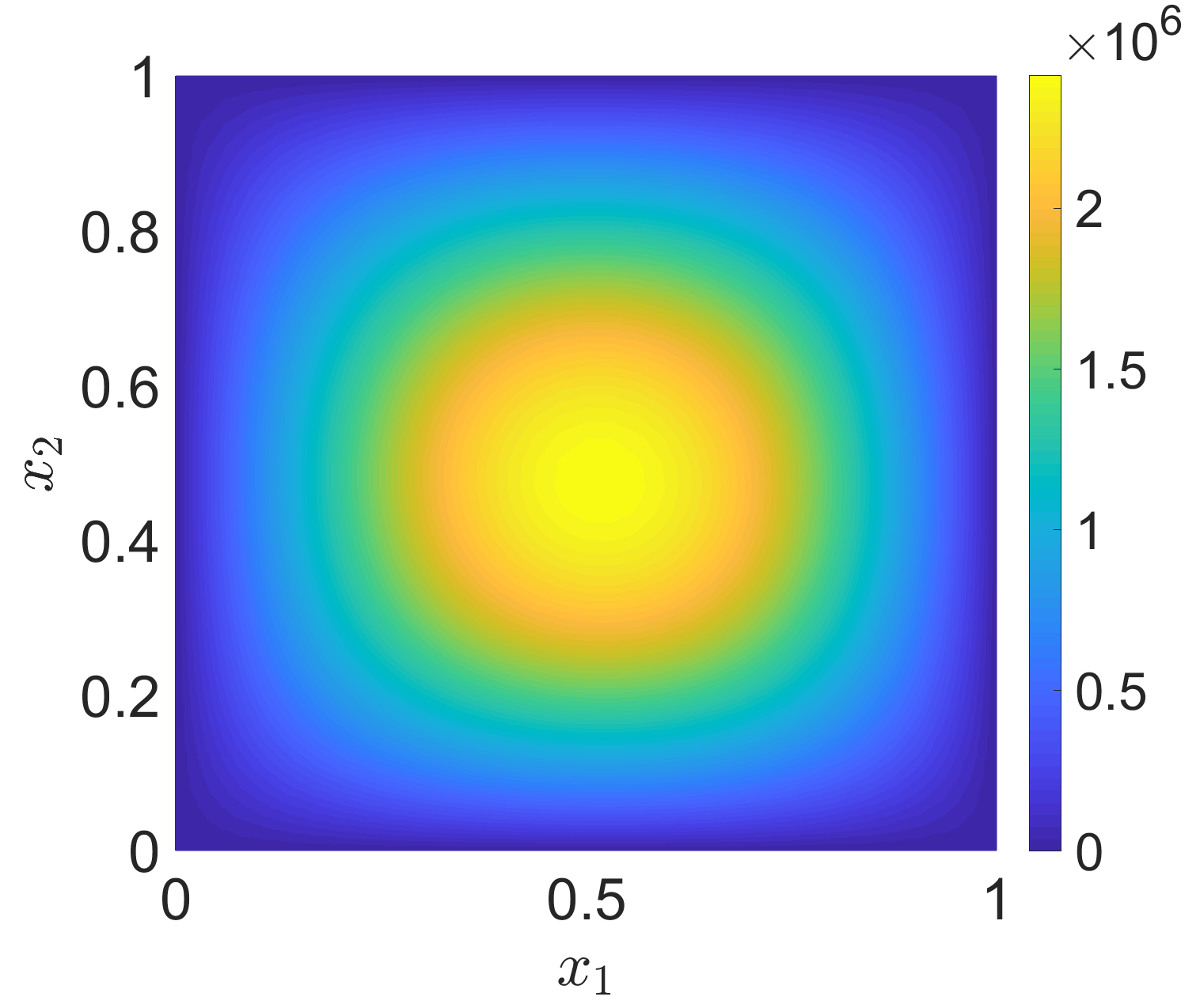}
    }
    \caption{ Several snapshots of the uncontrolled state }
    \label{Figa}
\end{figure}
\begin{example}
\label{exp1}
In this example, we ran Algorithm \ref{RHA} for the $\ell_2$-norm control cost with $\beta = 1000$, and different values of the prediction horizon $T$ with fixed $\delta =0.25 $.  Here the set of actuators consists of four actuators (indicator function), whose supports are specified in Figure \ref{Fig1}. The control domain $\omega =\cup^4_{i=1}R_i$ covers only 8 percent of  the domain. The corresponding numerical results are gathered in Table \ref{table1}.
\begin{table}[htbp]
\begin{center}
\scalebox{0.7}{
  \begin{tabular}{ | c | c | c | c | c | c | c |}
    \hline
    Prediction Horizon & $J^p_{T_{\infty}}$ & $\|y_{rh}\|_{L^2(0,T_{\infty};V)}$  & $\|y_{rh}(T_{\infty})\|_{V}$ &$\|y_{rh}(T_{\infty})\|_{H}$ & iter\\
    \hline
    $T = 1.5$ & $4.04\times 10^{1}$ &$ 8.59$ &$ 1.11 \times 10^{-5}$ & $ 1.50 \times 10^{-6}$& $2873$ \\
    \hline
    $T = 1$ &$4.35\times 10^{1}$ &$8.98$&$ 7.03\times10^{-4}$  &$ 9.83\times10^{-5}$ &$2411$ \\
    \hline
    $T = 0.75$ & $5.95\times 10^{1}$ &$1.06\times 10^{1}$& $5.38 \times 10^{-2}$ & $8.20 \times 10^{-3}$&$2046$ \\
    \hline
    $T = 0.5$ & $8.72\times 10^{2}$ &$4.12\times 10^{1}$& $1.50\times 10^{1} $ & $2.75$&$1649$ \\
    \hline
    $T = 0.25$ &$1. 97\times 10^{7}$& $6.24 \times 10^{3}$&  $6.96 \times 10^{3}$ &$1.50 \times 10^{3}$&$1063$ \\
    \hline
 \end{tabular}}
 \end{center}
 \caption{Numerical results for Example \ref{exp1}}
 \label{table1}
\end{table}
Moreover, Figures \ref{Fig3} and \ref{Fig4} demonstrate the logarithmic evolution of the spatial norm of the RH states  with respect to the different norms $H$ and $V$,  and  for different choices of $T$.   From Figures \ref{Fig3} and \ref{Fig4} and Table \ref{table1}, it can be observed that the RH state for the choices $T \in \{0.25, 0.5\}$ is exponentially unstable ($T^*>0.5$),  whereas for $T \in \{1,1.25,1.5\}$, it is exponentially stabilizing.  For the case $T=0.75$, it seems that RH state is stable but not asymptotically stable. Moreover, for every choice of $T$, the exponential rates for both  norms $H$ and $V$ are equal. By comparing the numerical results, we can conclude that the larger $T$ was chosen,  the better the performance  of RHC was achieved. However,  a larger prediction horizon $T$ leads to  a larger number of overall iterations.
 \begin{figure}[htbp]
    \centering
 \subfigure[ Example \ref{exp1} ]
    {
     \label{Fig1}
        \includegraphics[height=2.5cm,width=2.5cm]{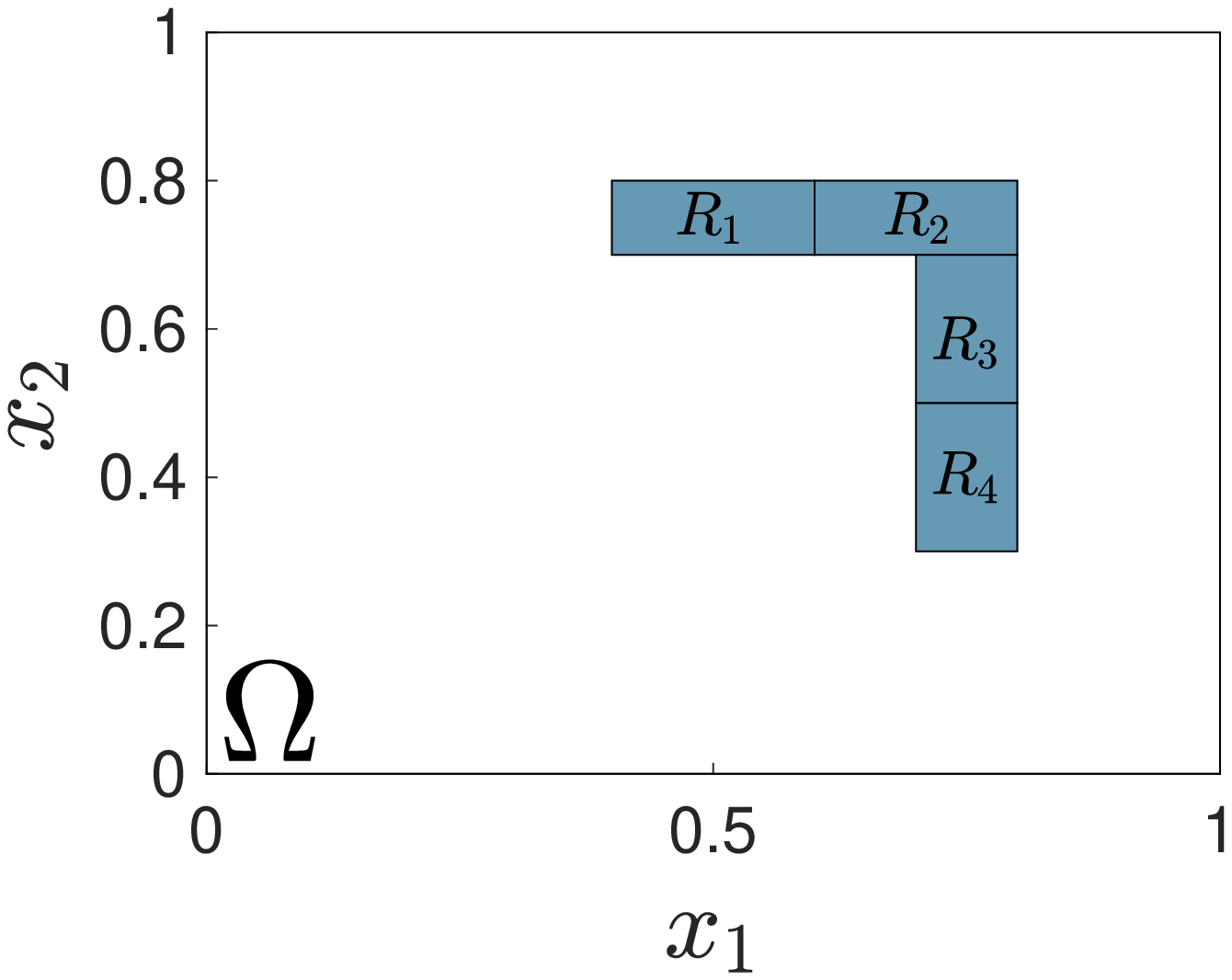}
    }
    \subfigure[ Example \ref{exp2} ]
    {
    \label{Fig2}
        \includegraphics[height=2.5cm,width=2.5cm]{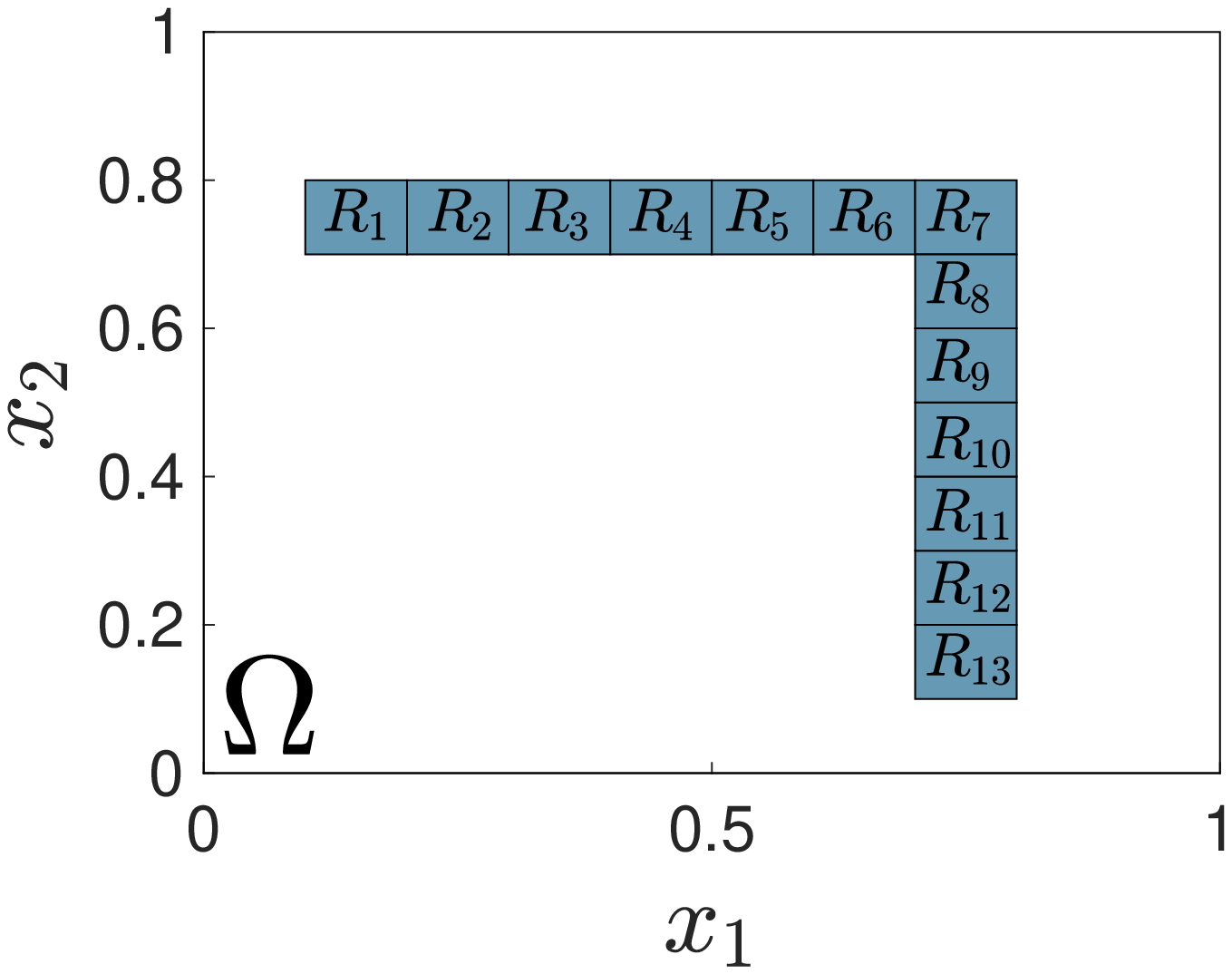}
    }
     \caption{Control domains}
\end{figure}
The logarithmic evolution for the absolute value of the RH controllers for the choices $T \in \{1.5,0.5\}$  are plotted in Figures \ref{Fig5} and \ref{Fig6}. As expected the corresponding RH controllers are more regular, if the ratio of prediction horizon $T$ to sampling time  $\delta$ is large.
\begin{figure}[htbp]
    \centering
    \subfigure[$t=0$]
    {
       \includegraphics[height=3cm,width=4cm]{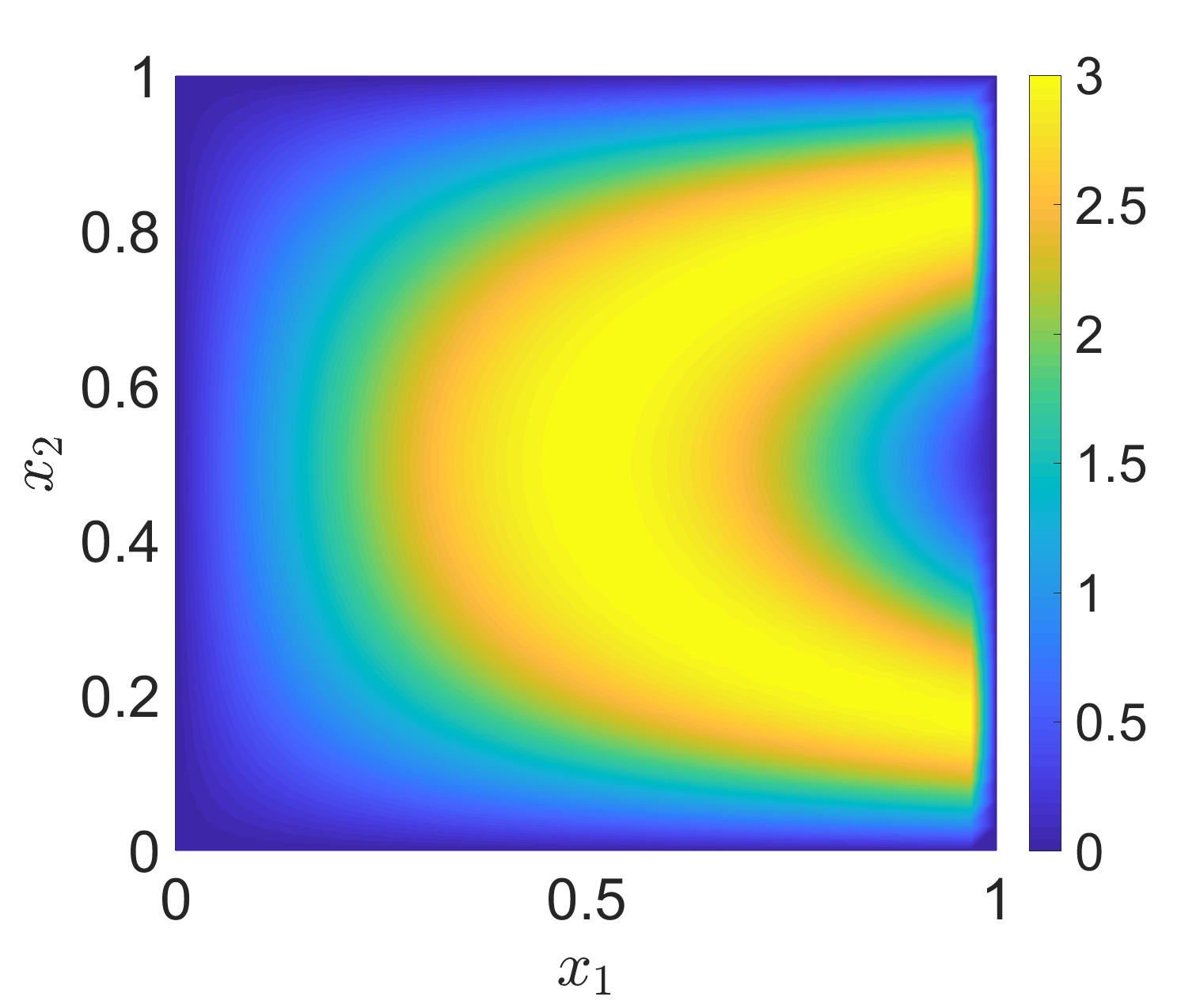}
    }
    \subfigure[ $t =5$]
    {
        \includegraphics[height=3cm,width=4cm]{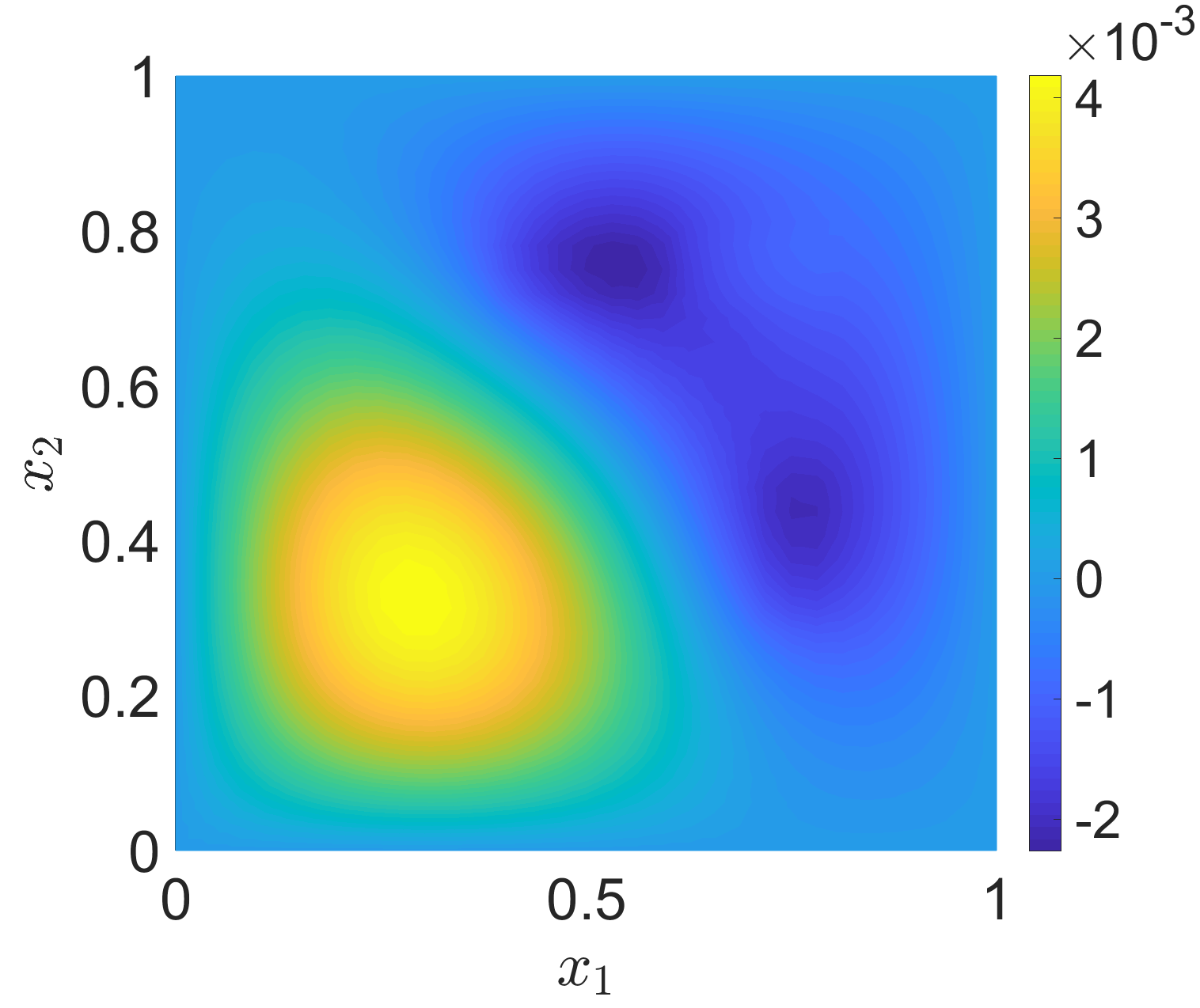}
    }
    \subfigure[ $t=10$]
    {
        \includegraphics[height=3cm,width=4cm]{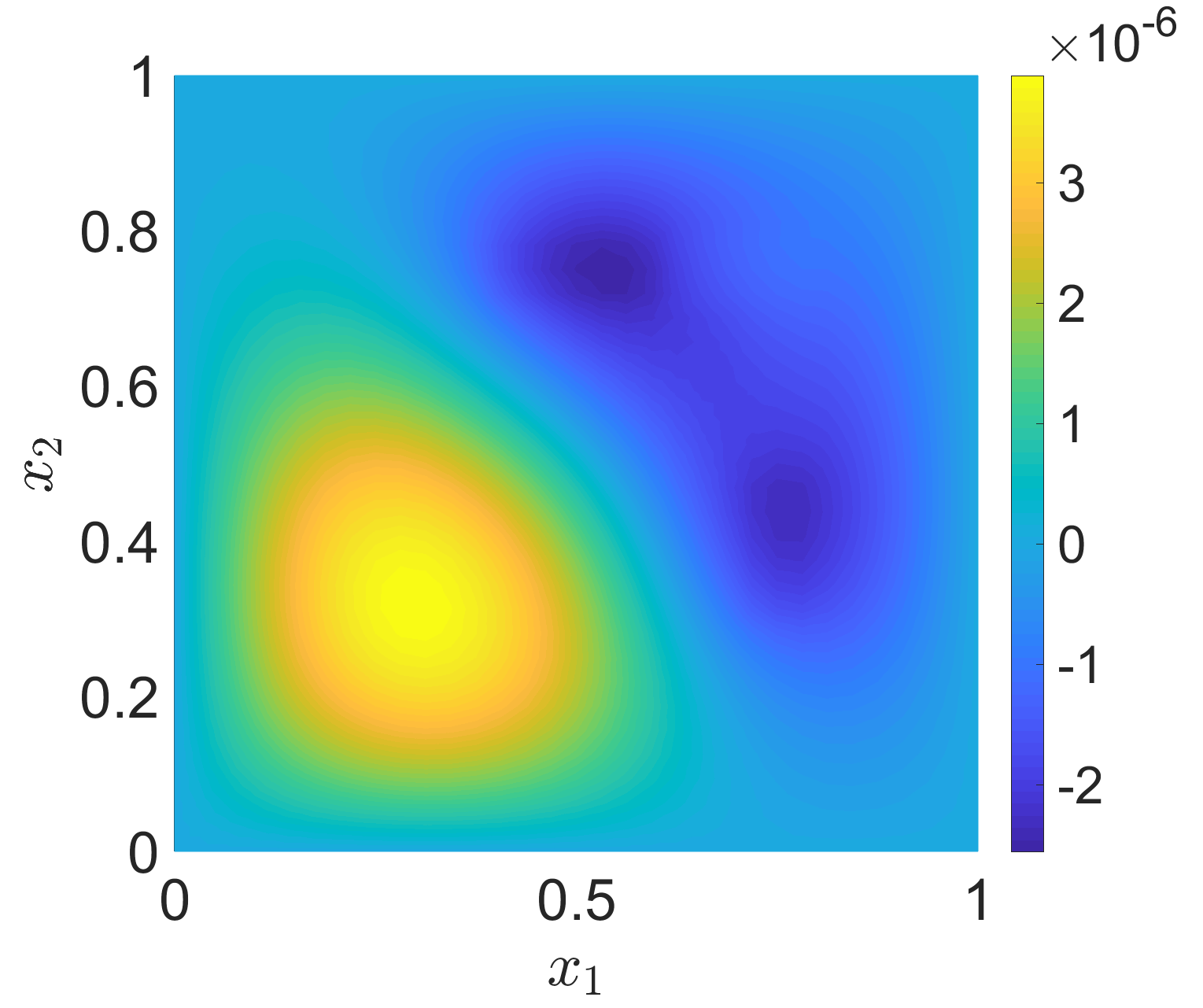}
    }
     \caption{Several snapshots of the RH state for the choice of $T = 1.5$  corresponding to Example \ref{exp1}}
    \label{Figb}
\end{figure}
Figure \ref{Figb} shows the RH state at different times for the choice of $T = 1.5$.
\begin{figure}[htbp]
    \centering
 \subfigure[$\mathcal{X}=V$ ]
    {
     \label{Fig3}
        \includegraphics[height=4cm,width=5cm]{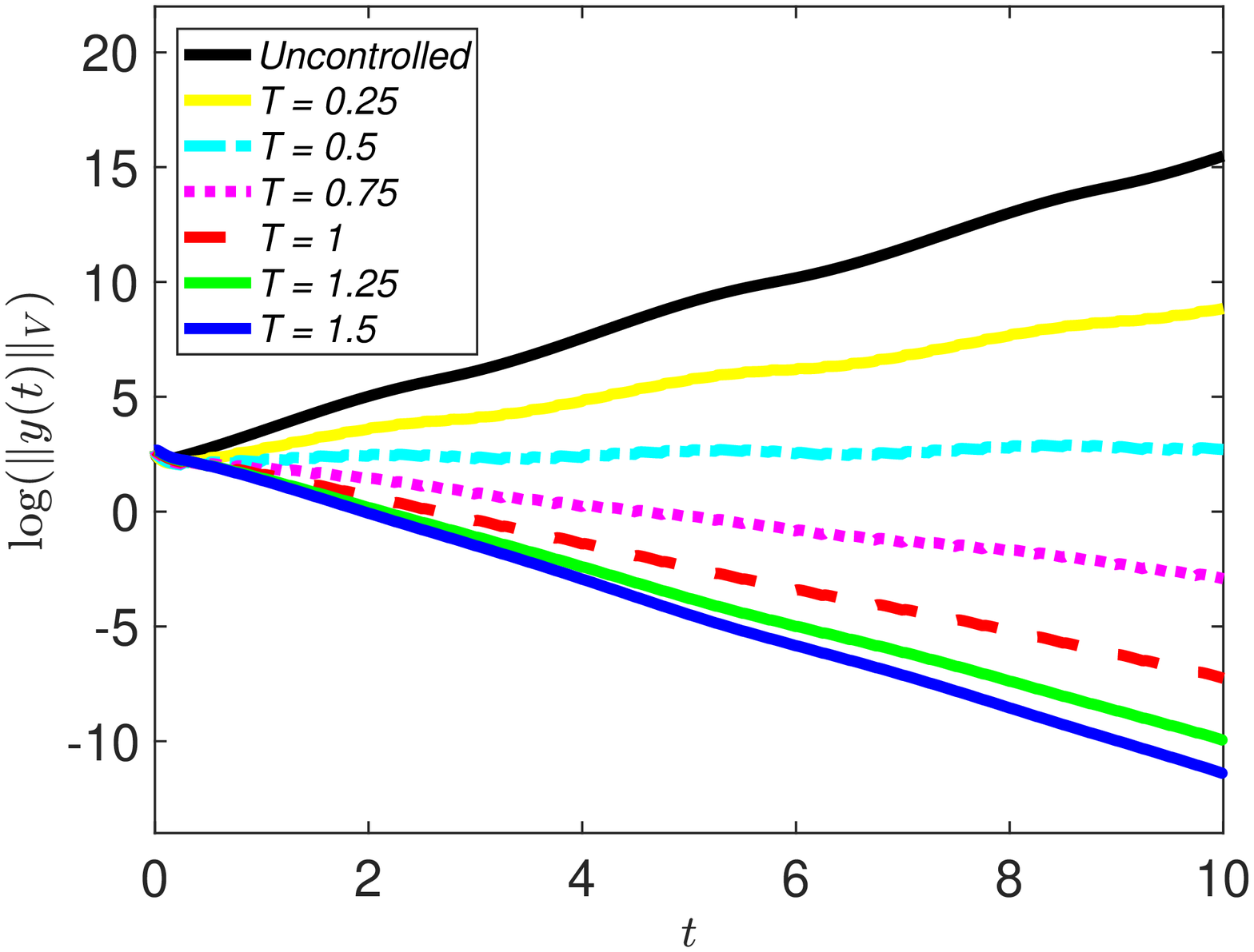}
    }
    \subfigure[$\mathcal{X}=H$ ]
    {
    \label{Fig4}
        \includegraphics[height=4cm,width=5cm]{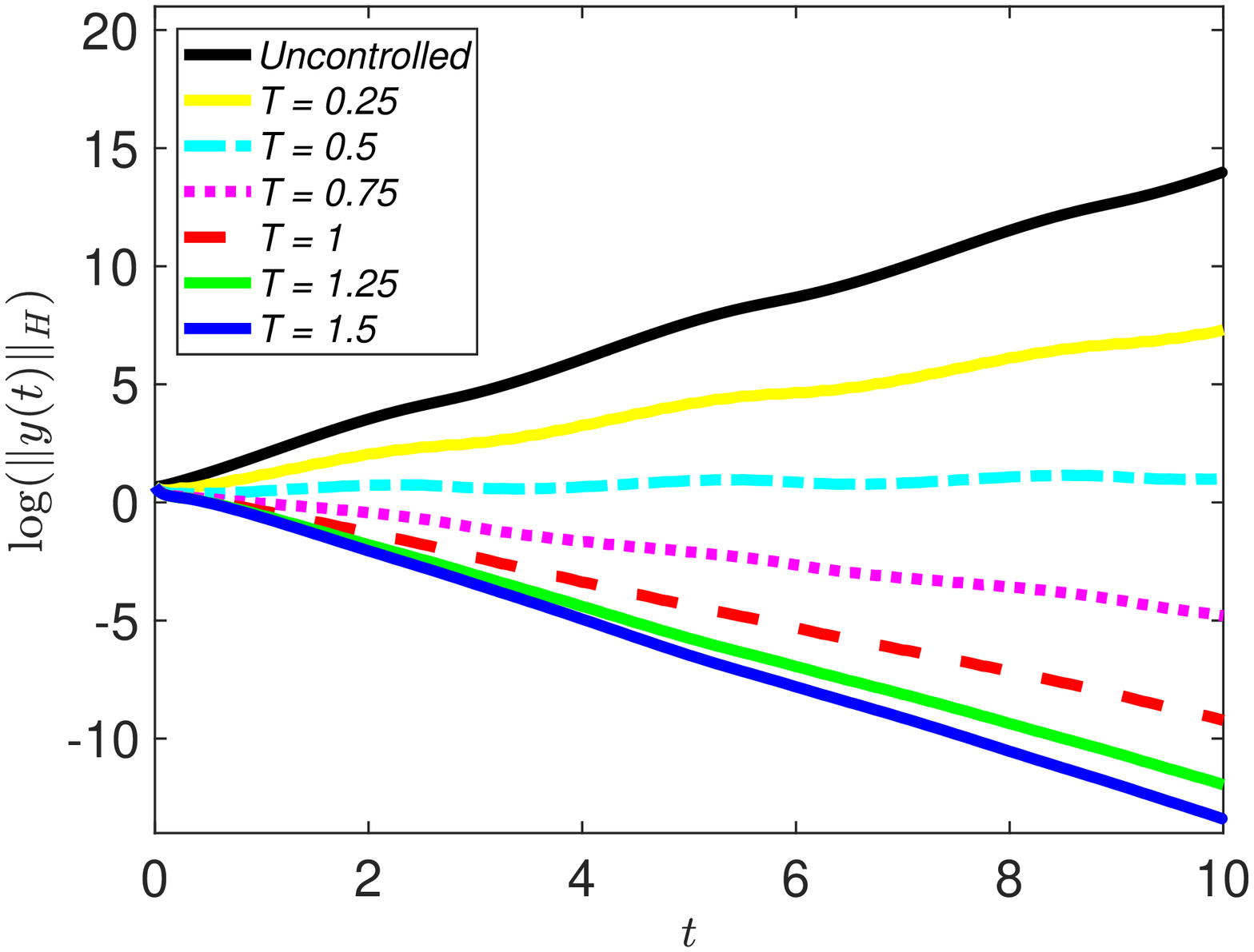}
    }
  \caption{Evolution of $\log(\|y_{rh}(t)\|_{\mathcal{X}})$ corresponding to Example \ref{exp1} for different choices of $T$ and $\mathcal{X}$}
\end{figure}

\begin{figure}[htbp]
    \centering
\subfigure[$T =1.5$ ]
    {
    \label{Fig5}
        \includegraphics[height=5cm,width=6cm]{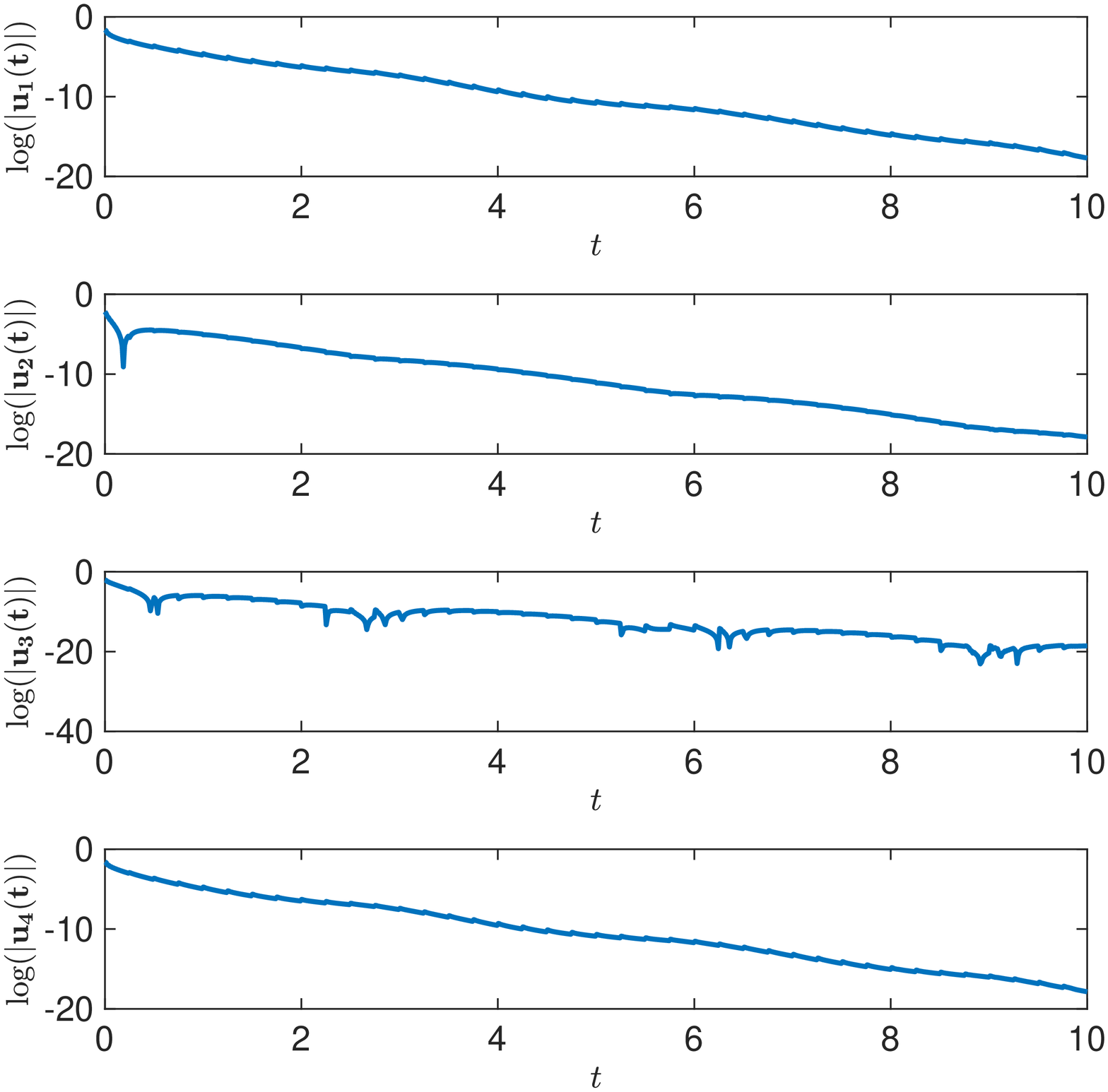}
    }
       \subfigure[$T =0.5 $]
    {
    \label{Fig6}
        \includegraphics[height=5cm,width=6cm]{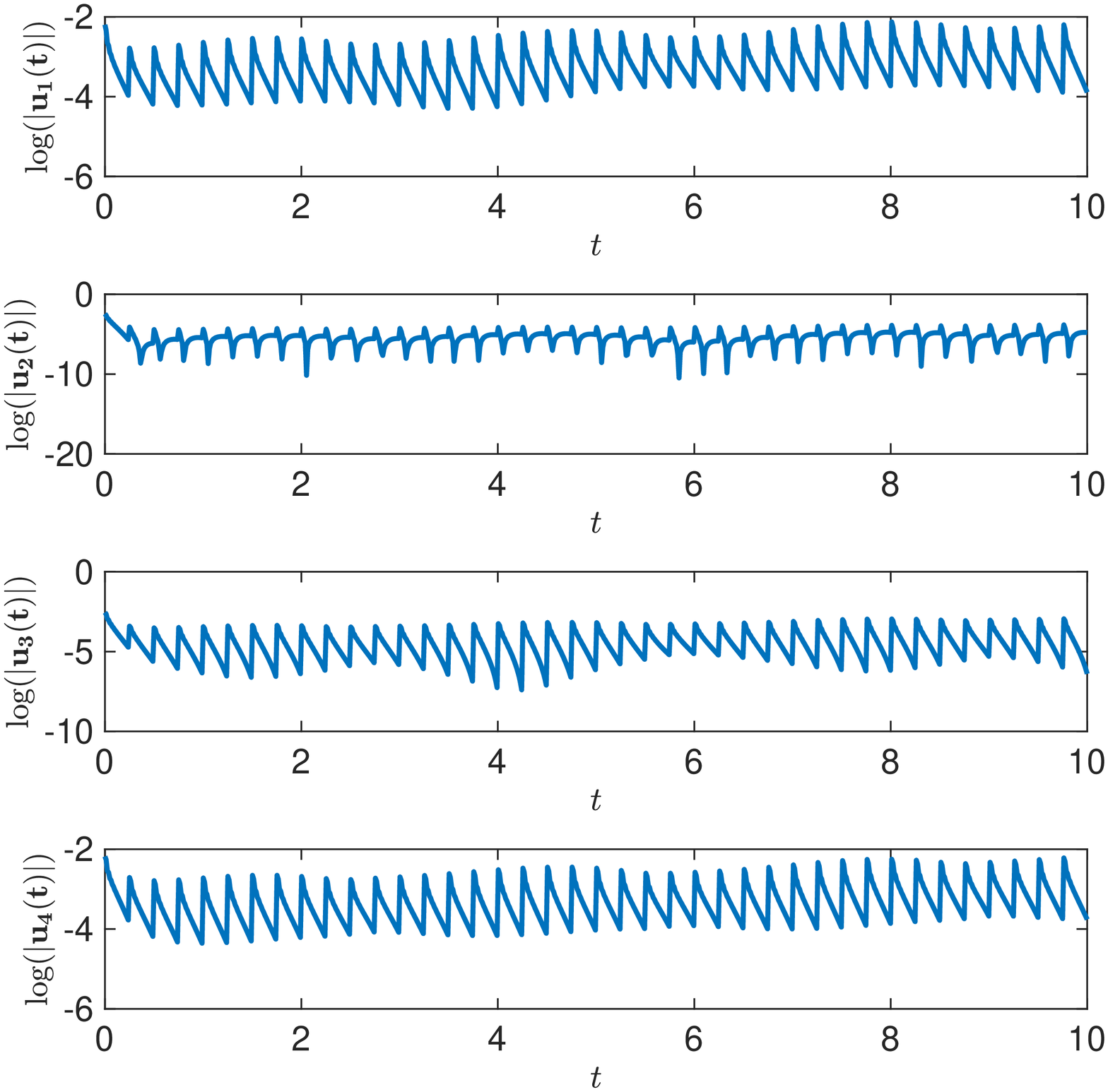}
    }
     \caption{Evolution of $\log(|(u_{rh})_i(t)|)$ corresponding to Example \ref{exp1} for $i =1,\dots,4$ and the choices $T =1.5 , 0.5 $}
\end{figure}
\end{example}

\begin{example}
\label{exp2}
In this example, we demonstrate the qualitative differences  between  the $\ell_1$- and  $\ell_2$-control costs. Here we set $\beta =5000$ and  considered 13 actuators, whose supports are specified in Figure \ref{Fig2}. Here the control domain consists of 13 percent of the domain. We ran algorithm \ref{RHA} for both of the control costs, different choices of  $T$ with fixed $\delta =0.25$. The corresponding  numerical results are summarized in Tables \ref{table2} and \ref{table3}.
\begin{table}[htbp]
\begin{center}
\scalebox{0.7}{
  \begin{tabular}{ | c | c | c | c | c | c | c |}
    \hline
    Prediction Horizon & $J_{T_{\infty}}$ & $\|y_{rh}\|_{L^2(0,T_{\infty};V)}$  & $\|y_{rh}(T_{\infty})\|_{V}$ &$\|y_{rh}(T_{\infty})\|_{H}$ & iter\\
    \hline
    $T = 1.5$ & $4.94\times 10^{1}$ &$ 7.76$ &$ 9.09 \times 10^{-6}$ & $ 1.24 \times 10^{-6}$& $1641$ \\
    \hline
    $T = 1$ &$5.34\times 10^{1}$ &$8.45$&$ 8.36\times10^{-4}$  &$ 1.19\times10^{-4}$ &$1262$ \\
    \hline
    $T = 0.75$ & $7.79\times 10^{1}$ &$1.08\times 10^{1}$& $9.76 \times 10^{-2}$ & $1.55 \times 10^{-2}$&$1008$ \\
    \hline
    $T = 0.5$ & $4.55\times 10^{3}$ &$8.88\times 10^{1}$& $4.86\times 10^{1} $ & $9.42$&$743$ \\
    \hline
    $T = 0.25$ &$6. 99\times 10^{8}$& $3.66 \times 10^{4}$&  $4.79 \times 10^{4}$ &$1.06 \times 10^{4}$&$521$ \\
    \hline
 \end{tabular}}
 \end{center}
 \caption{Numerical results corresponding to Example \ref{exp2} with  $\ell_2$-norm}
 \label{table2}
\end{table}
Moreover Figures  \ref{Fig7} and  \ref{Fig8} depict the evolution of $\log(\|y_{rh}(t)\|_{H})$ for different choices of $T$ and control costs. In both of the cases $\ell_1$-  and  $\ell_2$-norms, we can observe that RHC is exponentially stabilizing for $T$ large enough.  Clearly, the considerations concerning the value of $T$ from the previous example are also valid here.  Moreover, to obtain a rate of stabilization for the $\ell_1$-norm comparable to the $\ell_2$-norm, a larger value of $T$ is required.
\begin{table}[htbp]
\begin{center}
\scalebox{0.7}{
  \begin{tabular}{ | c | c | c | c | c | c | c |}
    \hline
    Prediction Horizon & $J_{T_{\infty}}$ & $\|y_{rh}\|_{L^2(0,T_{\infty};V)}$ & $\|y_{rh}(T_{\infty})\|_{V}$ &$\|y_{rh}(T_{\infty})\|_{H}$ & iter\\
    \hline
    $T = 2$ & $2.03\times 10^{2}$ &$ 9.65$ &$ 1.43 \times 10^{-4}$ & $ 2.00 \times 10^{-5}$& $1780$ \\
    \hline
    $T = 1.25$ &$3.68\times 10^{2}$ &$1.61\times10^{1}$&$ 8.87\times10^{-1}$  &$ 1.39\times10^{-1}$ &$777$ \\
    \hline
    $T = 1$ & $3.56\times 10^{4}$ &$1.79\times 10^{2}$& $1.23 \times 10^{2}$ & $2.27 \times 10^{1}$&$460$ \\
    \hline
    $T = 0.75$ & $8.07\times 10^{7}$ &$1.01\times 10^{4}$& $1.17\times 10^{4} $ & $2.47\times 10^{3}$&$419$ \\
    \hline
    $T = 0.5$ &$2. 86\times 10^{10}$& $2.24 \times 10^{5}$&  $3.26 \times 10^{5}$ &$7.25 \times 10^{4}$&$361$ \\
    \hline
 \end{tabular}}
 \end{center}
 \caption{Numerical results corresponding to Example \ref{exp2} with  $\ell_1$-norm}
 \label{table3}
\end{table}
Figures \ref{Fig9} and \ref{Fig10} depict the evolution of the absolute value of the RH controllers for the $\ell_1$-norm with $T =2$, and   for the  $\ell_2$-norm with $T =1.5$, respectively. As can be seen, while for the $\ell_2$-norm  all of the actuators were active (the corresponding controller were nonzero) consistently  over the whole interval $[0,T_{\infty}]$, for the case of  $\ell_1$-norm, not all of the actuator are active over  $[0,T_{\infty}]$. In particular,  the RH controllers  1, 7, and 8 were forced to be zero all  time,  the actuators 6 and 13  were active just for a very short interval at the beginning of the simulation, and the actuators 5 and 12 were also off for a short period of time.  We should mention that a similar behaviour was  also observed for different values of $T$.
\begin{figure}[htbp]
    \centering
  \subfigure[$\ell_1$-norm]
    {
    \label{Fig7}
        \includegraphics[height=4cm,width=5cm]{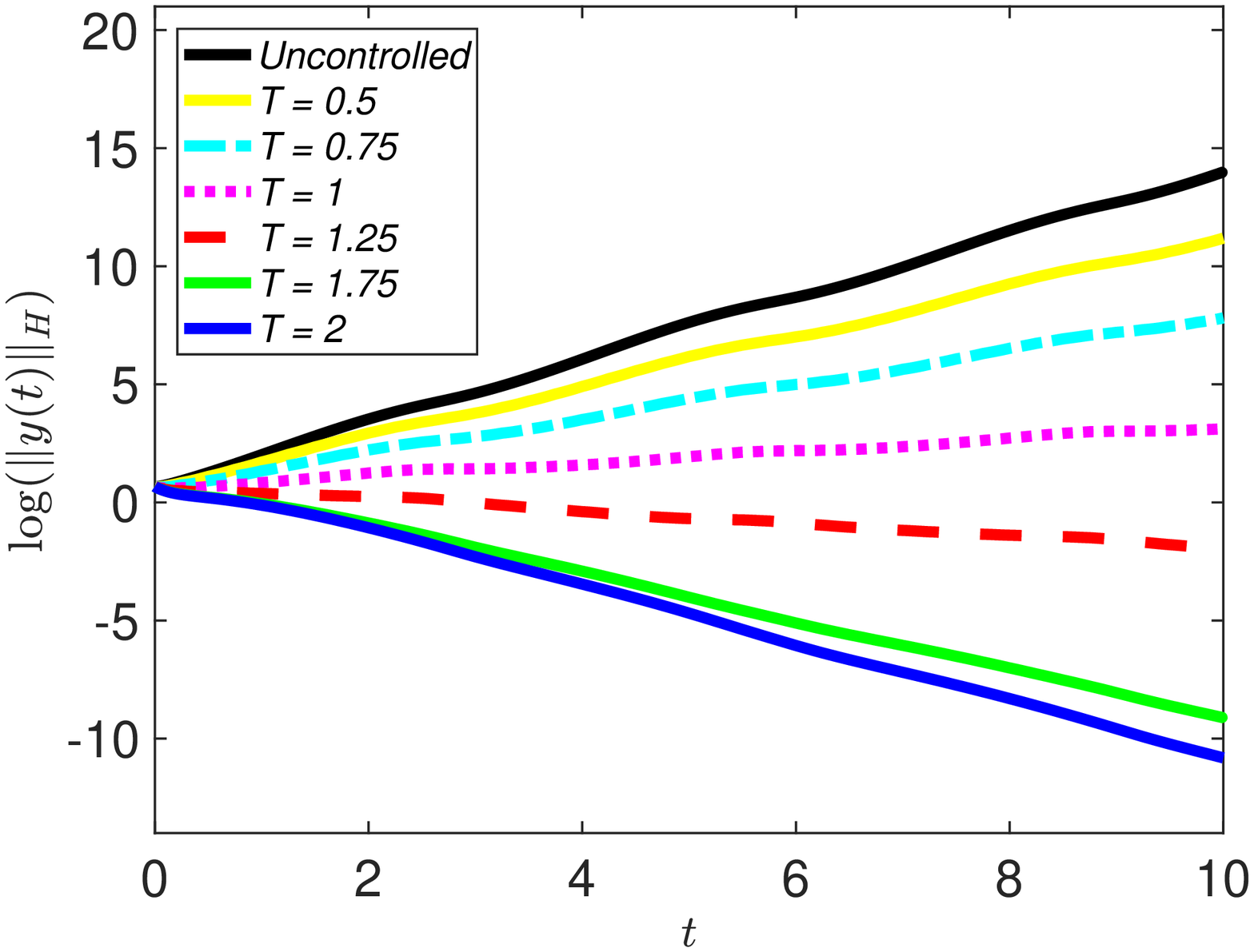}
    }
 \subfigure[$\ell_2$-norm ]
    {
    \label{Fig8}
        \includegraphics[height=4cm,width=5cm]{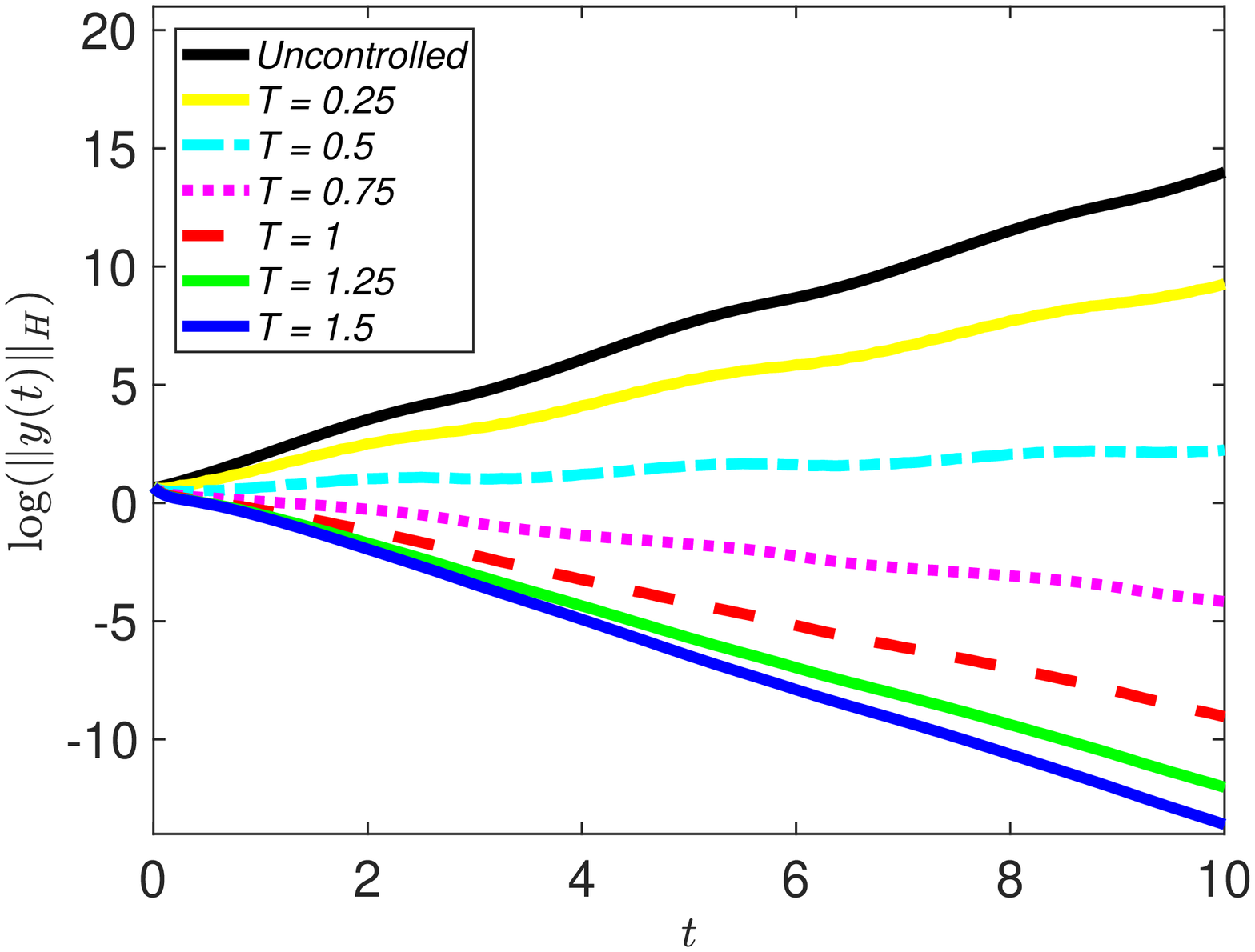}
    }
     \caption{Evolution of $\log(\|y_{rh}(t)\|_{H})$ corresponding to Example \ref{exp1} for different choices of $T$ and control costs ($\ell_2$-norm versus  $\ell_1$-norm) }
\end{figure}

\begin{figure}[htbp]
    \centering
 \subfigure[$\ell_1$-norm with $T=2$ ]
    {
     \label{Fig9}
        \includegraphics[height=7cm,width=7.5cm]{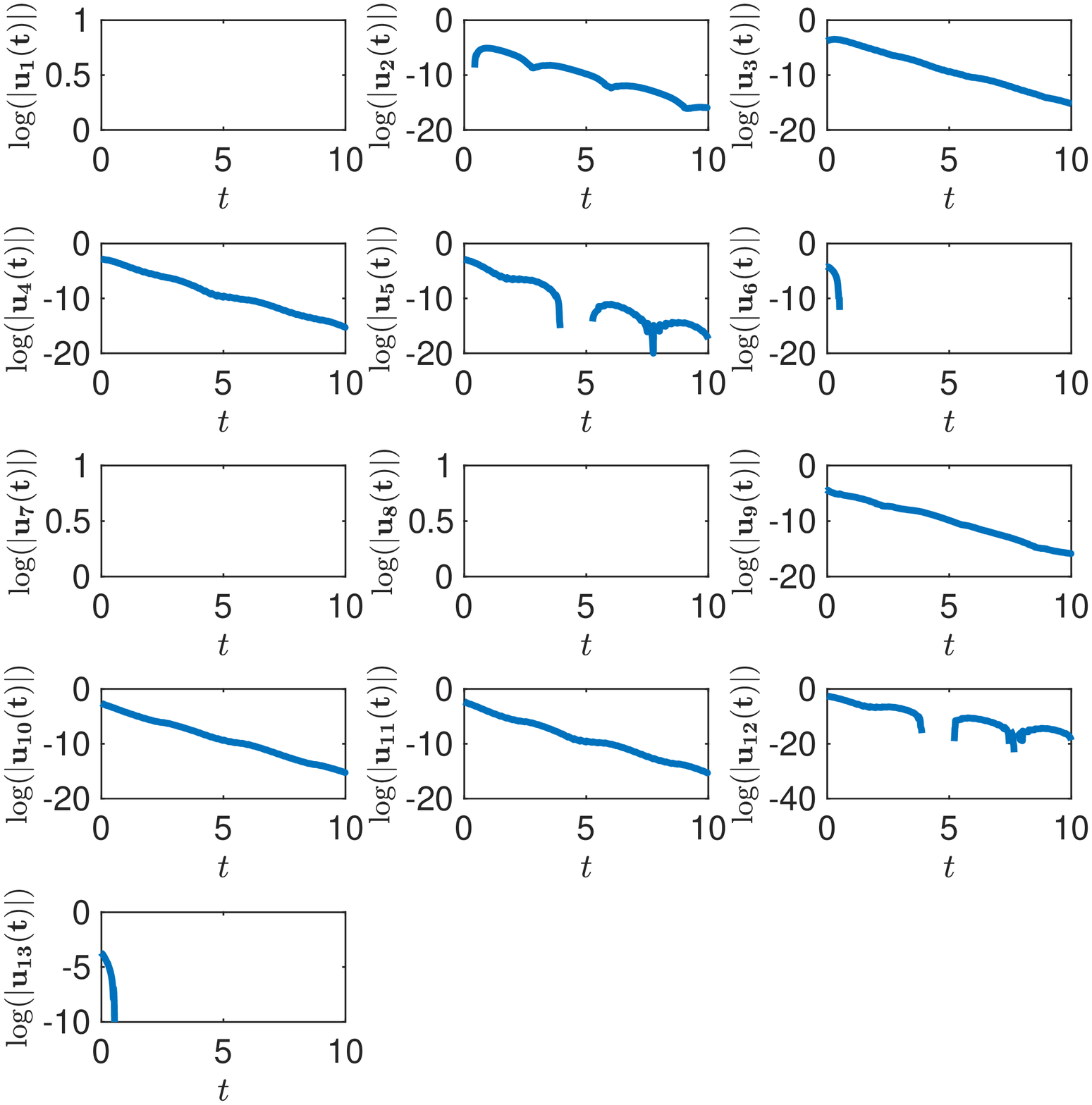}
    }
    \subfigure[$\ell_2$-norm with $T=1.5$ ]
    {
    \label{Fig10}
        \includegraphics[height=7cm,width=7.5cm]{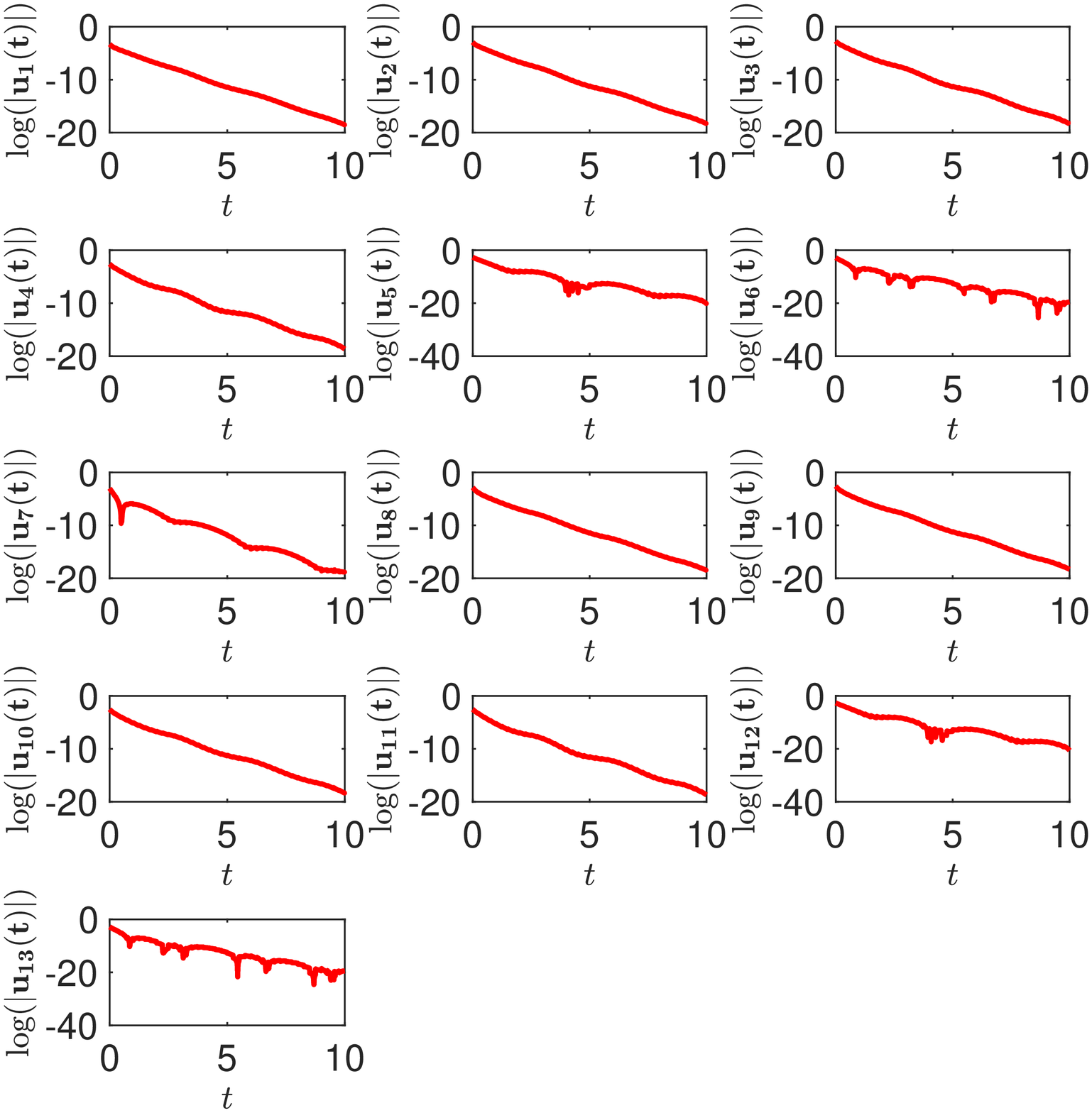}
    }
     \caption{Evolution of $\log(|(u_{rh})_i(t)|)$ corresponding to Example \ref{exp2} for $i =1,\dots,13$ and different control costs}
\end{figure}
\end{example}

Summarizing, for both numerical examples 
for a sufficiently large  prediction horizons $T\geq T^*>\delta$ the underlying system was successfully stabilized. Increasing $T$ leads to more efficient stabilization. 
On the other hand, the closer the prediction horizon $T$ is chosen to the sampling time $\delta$, the fewer overall iterations and computational effort is  required. Moreover, as desired, incorporating the squared  $\ell_1$-norm enhances stabilization in such a manner that at any time instance  fewer actuators are active.

\renewcommand\thesection{\Alph{section}}
\renewcommand\thesubsection{\thesection.\arabic{subsection}}
\setcounter{subsection}{0}
\setcounter{section}{0}
\section{Appendix}
\subsection{Proof of Theorem \ref{subopth}}
\label{apend1}
\begin{proof}
By Proposition \ref{pro1}, there exist a $T^*>0$ and $\alpha\in (0,1)$ such that for every $T\geq T^*$, $y_0 \in H$, and  $k \in \mathbb{N}$ with $k \geq 1$,  we have
 \begin{equation}
\label{e9}
V_T(t_k,y_{rh}(t_k)) - V_T(t_{k-1},y_{rh}(t_{k-1}))\leq-\alpha \int_{t_{k-1}}^{t_k} \ell(t, y_{rh}(t), \mathbf{u}_{rh}(t))dt,
\end{equation}
where $t_k = k\delta$ for $k =0,1,2,\dots$.   For any $k' \geq 1$, by summing inequality \eqref{e9} over  $k =1,2,\dots,k'$, we obtain
\begin{equation*}
V_T(t_{k},y_{rh}(t_{k'})) \leq V_T(0,y_0) - \alpha \int^{t_{k'}}_{0} \ell(t,y_{rh}(t),\mathbf{u}_{rh}(t))dt.
\end{equation*}
Taking the limit $k' \to \infty$ we can conclude the suboptimality inequality \eqref{ed27}.

Now we turn to inequality \eqref{ed28}. Using \eqref{e9} and the fact that $\delta < T$, we can write
\begin{equation}
\label{e10}
V_T(t_k,y_{rh}(t_k)) - V_T(t_{k-1},y_{rh}(t_{k-1})) \leq-\alpha V_{\delta}(t_{k-1},y_{rh}(t_{k-1})).
\end{equation}
 Moreover, due to \ref{P2} and \ref{P3}, for every $(t_0, y_0) \in \mathbb{R}_+ \times H$ we obtain
 \begin{equation}
\label{ed31}
V_{\delta}(t_0, y_0) \geq \gamma_1(\delta) \|y_0\|^2_{H}   \geq \frac{\gamma_1 (\delta)}{\gamma_2 (T)}V_{T}(t_0,y_0).
\end{equation}
Using \eqref{e10} and \eqref{ed31} we can write
\begin{equation}
\label{ed313}
V_T(t_k,y_{rh}(t_k)) \leq \left( 1- \frac{\alpha \gamma_1(\delta)}{\gamma_2(T)} \right) V_T(t_{k-1}, y_{rh}(t_{k-1}))  \text{ for every } k \geq 1.
\end{equation}
Since $0<\gamma_1(\delta)\leq \gamma_2(\delta)\leq \gamma_2(T)$ and $\alpha \in (0,1)$, we have $\eta := \left( 1- \frac{\alpha \gamma_1(\delta)}{\gamma_2(T)} \right) \in (0,1)$. Furthermore, by defining $\zeta := \frac{| \ln \eta |}{\delta}$, using Property \ref{P2} for $V_{T}(0,y_{0})$,  and Property \ref{P3} for $V_{T}(t_k, y_{rh}(t_k))$,  we can infer that
\begin{equation*}
\gamma_1(T) \|y_{rh}(t_k)\|^2_{H}\leq V_{T}(t_k,y_{rh}(t_k)) \leq  e^{-\zeta t_k}V_{T}(0,y_{0}) \leq e^{-\zeta t_k}\gamma_2(T)\|y_0\|^2_{H}
  \end{equation*}
for every $k \geq 1$. Hence, by setting $c'_H:= \frac{\gamma_2(T)}{\gamma_1(T)}$ we can write
\begin{equation}
\label{ed32}
\|y_{rh}(t_k)\|^2_{H} \leq  c'_He^{-\zeta k\delta}\|y_0\|^2_{H}   \text{ for every } k \geq 1.
\end{equation}
Moreover, for every $t >0$ there exists a $k \in \mathbb{N}$ such that $t \in [t_k, t_{k+1}]$. Using \eqref{Est1}-\eqref{e8}, and  \eqref{ed32},   we have for $t \in [t_k, t_{k+1}]$,
\begin{equation*}
\begin{split}
\|y_{rh}(t)\|^2_{H} &\stackrel{\text{\eqref{Est1}}}{\leq} c_{\delta}\left(\|y_{rh}(t_k)\|^2_{H}+\int^{t_{k+1}}_{t_k}\|\mathbf{u}_{rh}(t)\|^2_{U}dt\right)\\
& \stackrel{\text{\eqref{estiob}}}{\leq} c_{\delta}\left(\|y_{rh}(t_k)\|^2_{H}+\frac{1}{\alpha_{\ell}}V_T(t_k,y_{rh}(t_k))\right)\stackrel{\text{\eqref{e7}}}{\leq} c_{\delta}(1+\frac{\gamma_2(T)}{\alpha_{\ell}})\|y_{rh}(t_k)\|^2_{H}\\
& \stackrel{\text{\eqref{ed32}}}{\leq} c_{\delta}c'_H(1+\frac{\gamma_2(T)}{\alpha_{\ell}}) e^{-\zeta t_k}\|y_0\|^2_{H}\leq  c_{\delta}c'_H(1+\frac{\gamma_2(T)}{\alpha_{\ell}})\left( 1- \frac{\alpha \gamma_1(\delta)}{\gamma_2(T)}\right)^{-1}e^{-\zeta t_{k+1}}\|y_0\|^2_{H}\\
&\leq  c_{\delta}c'_H(1+\frac{\gamma_2(T)}{\alpha_{\ell}})\left( 1- \frac{\alpha \gamma_1(\delta)}{\gamma_2(T)} \right)^{-1}e^{-\zeta t}\|y_0\|^2_{H},
\end{split}
\end{equation*}
and therefore by setting
\begin{equation}
\label{e89}
c_H :=c_{\delta}c'_H(1+\frac{\gamma_2(T)}{\alpha_{\ell}})\left( 1- \frac{\alpha \gamma_1(\delta)}{\gamma_2(T)} \right)^{-1},
\end{equation}
we are finished with the verification of  \eqref{ed28} and the proof is complete.
\end{proof}

\subsection{Proof of Proposition \ref{Theo2}}
\label{apend2}
\begin{proof}
The existence result is standard and it will be obtained based on the Galerkin approximation,  using the eigenfunctions of the Laplacian as the basis functions and  a-priori estimates.  Therefore here, we omit the complete proof and restrict ourselves only to the derivation of the estimates \eqref{e29} and \eqref{e14}.  The rest of procedure is carried out in a similar manner an in e.g.,  \cite{MR1318914}[Chapter 1,
Section 3] and  \cite{MR769654}[Chapter 3, Sections 1.3, 1.4, and 3.2].  Before investing the estimates, we show that for every $a \in L^r(\Omega)$ with $r \geq n$ and   $\phi , \psi \in V$ we have
\begin{equation}
\label{e18}
\begin{cases}
\langle a \phi ,\phi \rangle_{V',V} \leq  c\|a\|_{L^r(\Omega)}\|\phi \|_H\| \phi\|_V   & \quad  \text{ for }  n \geq 1,        \\
\langle a \phi ,\psi\rangle_{V',V}\leq  c \|a\|_{L^r(\Omega)} (\|\phi \|_H\| \phi\|_V\|\psi\|_H\| \psi\|_V)^{\frac{1}{2}}& \quad\text{ for }  n \in\{1,2\}, \\
\langle a \phi ,\psi\rangle_{V',V}\leq c \|a\|_{L^r(\Omega)} \|\phi \|_H\| \psi\|_V  & \quad   \text{ for }  n \geq  3, \\
\end{cases}
\end{equation}
where $c>0$ is a generic constant and it depends only on $\Omega$.  For the case of  $n =1$, using the Agmon inequality,  we obtain
\begin{equation}
\label{e16}
\langle a \phi ,\psi \rangle_{V',V} \leq  c\|a\|_{L^r(\Omega)} \|\phi \|_{L^{\infty}(\Omega)}\| \psi\|_{L^{\infty}(\Omega)} \leq c \|a\|_{L^r(\Omega)} (\|\phi \|_H\| \phi\|_V\|\psi\|_H\| \psi\|_V)^{\frac{1}{2}}.
\end{equation}
Therefore for this case, the first and second inequalities in \eqref{e18} follow from \eqref{e16}. Moreover, since $V \hookrightarrow  L^4(\Omega)$ for the case $n=2$, by using an interpolation inequality (see e.g., \cite{MR0350177}),  we infer that
\begin{equation*}
\langle a \phi ,\psi \rangle_{V',V} \leq   c\|a\|_{L^r(\Omega)} \|\phi \|_{L^{4}(\Omega)}\| \psi\|_{L^{4}(\Omega)} \leq c \|a\|_{L^r(\Omega)} (\|\phi \|_H\| \phi\|_V\|\psi\|_H\| \psi\|_V)^{\frac{1}{2}},
\end{equation*}
and, as consequence, the first and second inequalities hold for  $n = 2$.  Finally \eqref{e18} for  $n \geq 3$  follows from the fact that  $V \hookrightarrow  L^{\frac{2n}{n-2}}(\Omega)$  and the following inequality
\begin{equation}
\label{e37}
\langle a \phi ,\psi \rangle_{V',V} \leq   c\|a\|_{L^n(\Omega)}  \| \phi\|_{H}\|\psi \|_{L^{\frac{2n}{n-2}}(\Omega)} \leq c \|a\|_{L^r(\Omega)}  \| \phi\|_{H}\|\psi \|_{V}.
\end{equation}
Now we turn to estimate \eqref{e29}. We assume that the solution $y$ to \eqref{e17} is regular enough. Then by multiplying the equation \eqref{e17} by $y(t)$, or equivalently by replacing  $\phi$ by $y(t)$ in the weak formulation \eqref{e19}, and using \eqref{e18},  we obtain for almost every $t \in (t_0,t_0+T)$ that
\begin{equation}
\label{e30}
\begin{split}
\frac{d}{2dt}\|y(t)\|^2_H& + \nu \|y(t)\|^2_V \leq  | \langle a(t)y(t),y(t) \rangle_{V',V}| + |(b(t)y(t), \nabla y(t))_H| + | \langle f(t),y(t) \rangle_{V',V}| \\
&\leq   cN(a,b) \|y(t) \|_H\| y(t)\|_V  +  \|f(t)\|_{V'}\|y(t)\|_{V}.
\end{split}
\end{equation}
Then from \eqref{e30} and using Gronwall's and Young's inequalities, we can infer that
\begin{equation}
\label{e30a}
\|y\|^2_{L^{\infty}(t_0,t_0+T;H)} + \nu\| y \|^2_{L^2(t_0,t_0+T;V)} \leq \exp(c^2N^2(a,b)T)\left( \|y_0\|^2_H + \| f\|^2_{L^2(t_0,t_0+T;V')} \right).
\end{equation}
where here the constant  $c$ depends also on $\nu$. Moreover, we can write
\begin{equation}
\label{e31}
\begin{split}
\|\partial_t y\|_{L^2(t_0,t_0+T;V')} &= \sup_{\|\phi \|_{L^2(t_0,t_0+T;V)}=1}   \int^{t_0+T}_{t_0} \langle  \partial_t y(t), \phi(t) \rangle_{V,V'}dt \\  &= \sup_{\|\phi \|_{L^2(t_0,t_0+T;V)}=1} \int^{t_0+T}_{t_0} \langle \nu\Delta y(t) - a(t)y(t)- \nabla \cdot (b(t)y(t)) + f(t), \phi(t) \rangle_{V',V}dt \\
&\leq  c(\nu+N(a,b)) \| y\|_{L^2(t_0,t_0+T;V)} + \|f\|_{L^2(t_0,t_0+T;V')},
\end{split}
\end{equation}
and, as a consequence,   \eqref{e29} follows from \eqref{e30a}  and \eqref{e31}.

 Finally, we come to the verification of the observability estimate \eqref{e14}.  Multiplying \eqref{e17} by $\frac{T+t_0-t}{T}y(t)$ and integrating in time from $t_0$ to $t_0+T$, we obtain
\begin{equation}
\label{e33}
\begin{split}
&\int^{t_0+T}_{t_0}\frac{t_0+T-t}{T} \langle  \partial_t y(t),y(t)\rangle_{V',V}= \\
&\int^{t_0+T}_{t_0}\frac{t_0+T-t}{T} \left( -\nu\|y(t)\|^2_{V}- \langle a(t)y(t),y(t) \rangle_{V',V} + (b(t)y(t), \nabla y(t))_H  +\langle f(t),y(t)\rangle_{V',V}  \right) \,dt.
\end{split}
\end{equation}
By integration by part, we can infer that
\begin{equation}
\label{e34}
\begin{split}
\int^{t_0+T}_{t_0}\frac{t_0+T-t}{T} \langle  \partial_t y(t),y(t)\rangle_{V',V} \,dt&= \int^{t_0+T}_{t_0}  \frac{t_0+T-t}{2T} \left( \frac{d}{dt}\|y(t)\|^2_H \right) \,dt \\& = \frac{1}{2T}\int^{t_0+T}_{t_0}\|y(t)\|^2_H\,dt - \frac{1}{2}\| y(t_0)\|^2_H.
\end{split}
\end{equation}
Now, using  \eqref{e18}, \eqref{e33}, \eqref{e34}, Young's inequality, and the fact that $\frac{T+t_0-t}{T} \leq 1$ for every  $t \in [t_0,t_0+T]$, we obtain
\begin{equation*}
\begin{split}
&\|y_0 \|^2_H =\frac{1}{T} \int^{t_0+T}_{t_0}\| y(t) \|^2_H\,dt\\
&+2\int^{t_0+T}_{t_0}\frac{t_0+T-t}{T}\left( \nu \|y(t)\|^2_{V}+ \langle a(t)y(t),y(t) \rangle_{V',V} - (b(t)y(t), \nabla y(t))_H -\langle f(t),y(t)\rangle_{V',V}  \right) \,dt \\
&\leq c_p\left(\frac{1}{T}+\frac{2\nu+1}{c_p}+ N(a,b)\right)\int^{t_0+T}_{t_0}\| y(t)\|^2_V\,dt +\int^{t_0+T}_{t_0}\| f(t)\|^2_{V'}\,dt,
\end{split}
\end{equation*}
where $c_p>0$ stands for the constant in  the Poincar\'e  inequality. This implies  \eqref{e14}.
\end{proof}

\subsection*{Acknowledgements}
The authors appreciate  Sergio S. Rodrigues for his  helpful comments and insights on the topic of finite-dimensional stabilizability  of the  time-varying parabolic equations. 

The work of K. Kunisch was partly supported by the ERC advanced grant 668998
(OCLOC) under the EU's H2020 research program.

\bibliographystyle{siam}
\bibliography{RHCforWAVE.bib}
\end{document}